\documentclass[12pt,a4paper]{article}

\newcommand{\bibcommenthead}{}
\usepackage{tensor}

\usepackage{amssymb,amsmath,amsthm,epsfig}
\numberwithin{equation}{section}
\numberwithin{figure}{section}
\numberwithin{table}{section}

\usepackage{mathrsfs}

\usepackage[table,xcdraw]{xcolor}

\usepackage{latexsym, enumerate}
\usepackage{eepic}
\usepackage{epic}
\usepackage{color}
\usepackage{ifpdf}

\usepackage[numbers,sort&compress]{natbib}
\usepackage{natbib}

\usepackage{comment}
\usepackage{titletoc}

\usepackage{dsfont}
\usepackage{multirow}
\usepackage{makecell}
\usepackage{bm}
\usepackage{epstopdf}
\usepackage{graphicx}
\usepackage{subfig}

\usepackage{tikz}

\graphicspath{{Figure/}}

\usepackage{hyperref}
\hypersetup{hypertex=true,
    colorlinks=true,
    linkcolor=blue,
    filecolor=blue,
    urlcolor=blue,
    citecolor=blue,
    bookmarksopen=true,
}
\usepackage{soul}
\usepackage{xcolor}
\usepackage{color}
\sethlcolor{green}
\soulregister\cite7 
\soulregister\citep7 
\soulregister\citet7 
\soulregister\ref7 
\soulregister\pageref7 

\usepackage[toc,page]{appendix}  

\usepackage{booktabs}
\usepackage{caption}

\usepackage{algorithm}
\usepackage{algpseudocode}

\theoremstyle{plain}
\newtheorem{lem}{Lemma}[section]
\newtheorem{thm}[lem]{Theorem}

\newtheorem{prop}[lem]{Proposition}

\theoremstyle{definition}

\theoremstyle{remark}
\newtheorem{rem}{Remark}[section]

\begin{document}
\title{	{\large\bf Multi-parameter determination in the semilinear Helmholtz equation} }

\author{
{Long-Ling Du}\thanks
{School of Mathematics, Hunan University, Changsha 410082, China. Email: 2089759515@qq.com}
\and
{Zejun Sun}\thanks
{Corresponding author. School of Mathematics, Hunan University, Changsha 410082, China. Email: sunzejun@hnu.edu.cn}
\and
{Li-Li Wang}\thanks
{School of Mathematics, Hunan University, Changsha 410082, China. Email: lilywang@hnu.edu.cn}
\and
{Guang-Hui Zheng}\thanks
{School of Mathematics, Hunan University, Changsha 410082, Hunan Province, China.
Email: zhenggh2012@hnu.edu.cn; zhgh1980@163.com}
}
\date{}
\maketitle

\begin{center}{\bf ABSTRACT}
\end{center}\smallskip
This paper studies an inverse boundary value problem for a semilinear Helmholtz equation with Neumann boundary conditions in a bounded domain $\Omega \subset \mathbb{R}^n$ ($n\ge2$). 
The objective is to recover the unknown linear and nonlinear coefficients from the associated Neumann-to-Dirichlet (NtD) map.
Using a higher-order linearization approach, we establish the unique determination of both coefficients from boundary measurements. 
For spatial dimensions $n\ge3$, uniqueness holds under $C^\gamma(\overline{\Omega})$ regularity assumptions with $0<\gamma<1$, while in the two-dimensional case uniqueness is obtained under Sobolev regularity $W^{1,p}(\Omega)$ with $p>2$. 
The analysis relies on the well-posedness of the forward problem together with techniques from linear inverse problems, including Runge-type approximation arguments and Fourier analysis.
In addition, we develop a numerical reconstruction framework for recovering the coefficients from boundary data. 
The forward problem is discretized using a finite difference scheme combined with a quasi-Newton iteration, and the inverse problem is formulated within a Bayesian inference framework. 
Posterior distributions of the coefficients are explored using the preconditioned Crank–Nicolson (pCN) Markov chain Monte Carlo algorithm, which provides both point estimates and uncertainty quantification.
Numerical experiments demonstrate the effectiveness of the proposed reconstruction method and illustrate the theoretical uniqueness results.

\smallskip
{\bf keywords}: Semilinear Helmholtz equation, inverse problem, Neumann-to-Dirichlet map, higher-order linearization, uniqueness, Bayesian inference, numerical reconstruction.

\section{Introduction}\label{sec1}

\textbf{Mathematical setup and main results.}
In this paper, we consider the semilinear Helmholtz equation with Neumann boundary condition: 
\begin{equation}\label{problem}
\begin{cases}
\Delta u + k^2(1+\alpha(x))u + k^2\beta(x)u^3 = 0 & \text{in } \Omega,\\
\dfrac{\partial u}{\partial \nu} = g(x) & \text{on } \partial\Omega,
\end{cases}
\end{equation}
where $\Omega \subset \mathbb{R}^n$ is a connected bounded domain with $C^{\infty}$ boundary, and $\nu$ denotes the unit outward normal vector on $\partial\Omega$. 
The wavenumber $k$ is real, and the coefficients $\alpha(x),\beta(x)\in C^\gamma(\overline{\Omega})$ with $0<\gamma<1$. 
The boundary data satisfies $g\in C^{1,\gamma}(\partial\Omega)$. 
The coefficients $\alpha(x)$ and $\beta(x)$ as well as the boundary data $g(x)$ are assumed to be real-valued, while the solution $u$ is complex-valued. 
We also assume that $0$ is not a Neumann eigenvalue of the operator $\Delta + k^2(1+\alpha(x))$.

Our goal is to determine the unknown coefficients $\alpha(x)$ and $\beta(x)$ from boundary measurements. 
To this end, we employ the higher-order linearization method to establish the unique recovery of these coefficients from the boundary data associated with problem \eqref{problem}, as stated in Theorems \ref{dl1} and \ref{dl2}.

We define the Neumann-to-Dirichlet (NtD) map by
\[
N_{\alpha,\beta}: C^{1,\gamma}(\partial\Omega)\rightarrow C^{2,\gamma}(\partial\Omega), 
\qquad g \mapsto u_g|_{\partial\Omega},
\]
where $u_g$ denotes the solution to \eqref{problem} corresponding to the Neumann data $g$.

%

\begin{thm}\label{dl1}
	Let $\Omega \subset \mathbb{R}^n$ ($n \geq 3$) be a connected bounded open domain with $C^\infty$ boundary. 
	Suppose $\alpha_1, \alpha_2, \beta_1, \beta_2 \in C^\gamma(\overline{\Omega})$. 
	If $N_{\alpha_1,\beta_1} = N_{\alpha_2,\beta_2}$, then $\alpha_1 = \alpha_2$ and $\beta_1 = \beta_2$ in $\Omega$.
\end{thm}

\begin{thm}\label{dl2}
	Let $\Omega \subset \mathbb{R}^2$ be a connected bounded open domain with $C^\infty$ boundary. 
	Assume $\alpha_1, \alpha_2, \beta_1, \beta_2 \in W^{1,p}(\Omega)$ with $p > 2$. 
	If $N_{\alpha_1,\beta_1} = N_{\alpha_2,\beta_2}$, then $\alpha_1 = \alpha_2$ and $\beta_1 = \beta_2$ in $\Omega$.
\end{thm}

%

\begin{rem}
For $n=2$, from the embedding theorem, if $0<\gamma\leq1-{2}/{p}$ and $p>2$, $W^{1,p}(\Omega)\hookrightarrow C^{\gamma}(\overline{\Omega})$. Therefore it finds that, to achieve uniqueness in the two-dimensional case, we need to impose stronger regularity conditions on the reconstructed parameters.
\end{rem}

\vspace{10pt}
\noindent\textbf{Background and motivation.}
Problem \eqref{problem} arises from the Kerr effect in nonlinear optics on a bounded domain 
$\Omega\subset\mathbb{R}^n$ ($n\ge2$) with smooth boundary. 
According to \cite{Born2023}, the scalar field $u$ satisfies the semilinear Helmholtz equation with Neumann boundary condition
\begin{equation}\label{problem come from}
\begin{cases}
\Delta u+k^2(1+\alpha(x))u+k^2\beta(x)\lvert u \rvert ^{2}u= 0 &\text{in } \Omega,\\
\dfrac{\partial u}{\partial \nu}=g(x) &\text{on } \partial\Omega,
\end{cases}
\end{equation}
where $\alpha(x)$ and $\beta(x)$ denote the linear and nonlinear susceptibilities describing the electromagnetic properties of the medium. 
If $u$ is restricted to be real-valued, equation \eqref{problem come from} reduces to the form of problem \eqref{problem} considered in this paper.

The higher-order linearization method has become an effective tool for studying inverse problems for nonlinear partial differential equations. 
In \cite{Aremark2020}, Krupchyk and Uhlmann investigated the Dirichlet problem
\[
-\Delta u+V(x,u)=0, \qquad 
V(x,u)=\sum_{m=2}^{\infty}V_m(x)\frac{u^m}{m!},
\]
and proved that the coefficients $V_m(x)$ can be uniquely determined from the partial Dirichlet-to-Neumann (DtN) map. 
Later, in \cite{Partialdata2020}, they considered
\[
-\Delta u+q(x)(\nabla u)^2+V(x,u)=0,
\qquad 
V(x,u)=\sum_{m=3}^{\infty}V_m(x)\frac{u^m}{m!},
\]
and showed that both $q(x)$ and $V_m(x)$ are uniquely determined from partial DtN data. 
For equations on Riemannian manifolds, Feizmohammadi and Oksanen \cite{k=12020} studied
\[
-\Delta_g u + V(x,u)=0, \qquad 
V(x,u)=\sum_{m=1}^{\infty}V_m(x)\frac{u^m}{m!},
\]
and established the unique recovery of $V_m(x)$ from the DtN map. 
Higher-order linearization techniques have also been successfully applied to inverse problems for nonlinear hyperbolic equations 
\cite{Inversehyperbolic1,Inversehyperbolic2,Inversehyperbolic3} 
and to coupled systems arising in mean field game theory 
\cite{Liu1,Liu2,Liu3,Liu4}.

For semilinear Helmholtz equations, several uniqueness results have been obtained in recent years. 
Lu \cite{HelmholtzLu} proved that for the Dirichlet problem
\[
-\Delta u-k^2u+q(x)u^2=0,
\]
the coefficient $q(x)$ can be uniquely determined by the Dirichlet-to-Neumann map. 
This result was later extended in \cite{Semielliptic2023} to the equation 
\[
(-\Delta_g+V)u+qu^2=0
\]
on a Riemannian manifold. 
Moreover, \cite{Recover2024} studied a coupled semilinear Helmholtz system with quadratic nonlinearities and showed that multiple coefficients can be uniquely recovered from internal measurements. 
From a computational perspective, \cite{Born2023} developed numerical algorithms for recovering piecewise constant coefficients $\alpha$ and $\beta$ in the Dirichlet problem.

Motivated by these developments, our goal is to complement the uniqueness theory for the Neumann inverse boundary value problem associated with
\[
\Delta u+k^2(1+\alpha(x))u+k^2\beta(x)u^3=0.
\]
In particular, we prove that the coefficients $\alpha(x)$ and $\beta(x)$ can be uniquely determined from the Neumann-to-Dirichlet map $N_{\alpha,\beta}$ using a higher-order linearization approach. 
Furthermore, we develop a numerical reconstruction framework for recovering these coefficients from boundary measurements. 
The forward problem is discretized using a finite difference scheme together with a quasi-Newton iteration, while the inverse problem is addressed within a Bayesian inference framework with posterior sampling.

\vspace{10pt}
\noindent\textbf{Technical developments and discussion.}
This paper investigates the unique recovery of the linear and nonlinear coefficients in the semilinear Helmholtz equation from boundary measurement data. 
As a first step, we establish the well-posedness of the forward problem for the Neumann boundary value problem by applying the implicit function theorem through the construction of suitable holomorphic mappings. 

To analyze the inverse problem, we first study the corresponding linearized inverse problem. Due to dimensional restrictions in the construction of complex geometrical optics (CGO) solutions, the analysis is divided into two cases. 
For the two-dimensional case ($n=2$), a uniqueness result is obtained using existing results in \cite{2D2015}. 
For dimensions $n\ge3$, we construct CGO solutions for the linear Helmholtz equation and prove their existence using the Hahn–Banach theorem. 
Combined with Runge-type approximation arguments and Fourier analysis techniques, this yields the uniqueness of the associated linear inverse problem. 
Based on these density results, the higher-order linearization method is then applied to establish the unique determination of both the linear and nonlinear coefficients in the semilinear Helmholtz equation from the Neumann-to-Dirichlet map.

In addition to the theoretical analysis, we develop a numerical reconstruction framework for recovering the unknown coefficients from boundary measurements. 
The forward problem is discretized using a finite difference scheme and solved by a quasi-Newton iteration, and the inverse problem is formulated within a Bayesian inference framework. 
Posterior distributions of the coefficients are explored using sampling-based methods (i.e., pCN algorithm \cite{Dashti2015}),  enabling the practical reconstruction of the unknown parameters from noisy boundary data.

The rest of the paper is organized as follows. 
Section~\ref{sec2} establishes the well-posedness of the forward problem. 
Section~\ref{sec3} studies the associated linear inverse problem. 
Section~\ref{sec4} proves the uniqueness results for the nonlinear inverse problem using higher-order linearization. 
Section~\ref{sec5} presents the reconstruction method together with several numerical examples. 
Finally, Section~\ref{sec6} concludes the paper.

\section{Well-posedness of the forward problem}\label{sec2}
In this section, we prove the following theorem for the well-posedness of problem (\ref{problem}).

\begin{thm}\label{sdx}
If 0 is not the Neumann eigenvalue of $\Delta+k^2(1+\alpha(x))$, then exist $\delta>0$, $C>0$ such that for any $g\in B_{\partial\Omega}:=\{g\in C^{1,\gamma}(\partial\Omega):\lVert g \rVert_{C^{1,\alpha}(\partial\Omega)}<\delta\}$, the problem (\ref{problem}) has a solution $u=u_{g}\in C^{2,\gamma}(\overline\Omega)$, which satisfies
\begin{center}
	$\lVert u \rVert_{C^{2,\gamma}(\overline\Omega)}\leq C 	\lVert g \rVert_{C^{1,\gamma}(\partial\Omega)}$.
\end{center}
\noindent
Furthermore, the solution $u$ is unique within the class $\{u\in C^{2,\gamma}(\overline{\Omega}):\lVert u \rVert_{C^{2,\gamma}(\overline\Omega)}\leq C\delta\}$ and $u$ is depends holomorphically on $g\in B_{\delta}(\partial\Omega)$.
\end{thm}

\begin{proof}
We use the implicit function theorem to prove.

Let
$B_{1}=C^{1,\gamma}(\partial\Omega)$, $B_{2}=C^{2,\gamma}(\overline\Omega)$, $B_{3}=C^{\gamma}(\overline\Omega)\times\ C^{1,\gamma}(\partial\Omega)$. Consider the map
\begin{center}
	$F:B_{1}\times B_{2} \rightarrow B_{3}$,
\end{center}
\begin{center}
	$F(g,u)=(\Delta u+k^2(1+\alpha(x))u+k^2\beta(x)u^3, \dfrac{ \partial u }{ \partial \nu }\Big\vert_{\partial\Omega}-g)$.
\end{center}

Firstly, we prove that $F$ is well-defined because
\begin{center}
	$\lVert uv \rVert_{C^{m,\gamma}(\overline\Omega)}\leq C(\lVert u \rVert_{C^{m,\gamma}(\overline\Omega)}\lVert v \rVert_{L^{\infty}(\Omega)}+\lVert v \rVert_{C^{m,\gamma}(\overline\Omega)}\lVert u \rVert_{L^{\infty}(\Omega)}), $
\ $u$, $v\in C^{m,\gamma}(\overline\Omega)$
\end{center}
\noindent

According to the definition of the space $C^{m,\gamma}(\overline\Omega)$, $uv\in C^{m,\gamma}(\overline\Omega)$. This indicates that the space $C^{m,\gamma}(\overline\Omega)$ is closed under multiplication operations. Consequently $k^2(1+\alpha(x))u$, $k^2\beta(x)u^3\in C^{\gamma}(\overline\Omega)$. Therefore, $\Delta u+k^2(1+\alpha(x))u+k^2\beta(x)u^3\in C^{\gamma}(\overline\Omega)$.
\vspace{1mm}

Because of $u\in C^{2,\gamma}(\overline\Omega)$, $\dfrac{ \partial u }{ \partial \nu }\Big\vert_{\partial\Omega}\in C^{1,\gamma}(\partial\Omega)$ and $\dfrac{ \partial u }{ \partial \nu }\Big\vert_{\partial\Omega}-g\in C^{1,\gamma}(\partial\Omega)$.
\vspace{1mm}

Thus, we have proven that the space mapped to $F$ is consistent with its definition, and $F$ is a map, therefore F is well-defined

Secondly, we prove that $F$ is holomorphic. Because F is locally bounded, according to \cite[p.133, Theorem 1]{Spectral1987}, it is sufficient to prove that F is weakly holomorphic.

For this purpose, let any $(g_{0},u_{0})$, $(g,u)\in B_{1}\times B_{2}$, and we prove that
\begin{center}
	$l:\mathbb{C}\rightarrow B_{3}$,
\end{center}
\begin{center}
	$\lambda\longmapsto F((g_{0},u_{0})+\lambda(g,u))$,
\end{center}
is holomorphic on $\mathbb{C}$.

Because
$$\lim_{h\rightarrow 0,h\in\mathbb{C}}\dfrac{l(\lambda+h)-l(\lambda)}{h}$$
$$\begin{aligned}
	&=\lim_{h\rightarrow 0,h\in\mathbb{C}}\dfrac{F((g_{0},u_{0})+(\lambda+h)(g,u))-F((g_{0},u_{0})+\lambda(g,u))}{h}\\
	&=\lim_{h\rightarrow 0,h\in\mathbb{C}}\dfrac{F(g_{0}+(\lambda+h)g,u_{0}+(\lambda+h)u)-F(g_{0}+\lambda g,u_{0}+\lambda u)}{h}\\
	&=\lim_{h\rightarrow 0,h\in\mathbb{C}}\dfrac{\Big(h\Delta{u}+hk^2(1+\alpha(x))u+k^2\beta(x)[(u_{0}+(\lambda+h)u)^3-(u_{0}+\lambda u)^3],h\dfrac{ \partial u }{ \partial \nu }\Big\vert_{\partial\Omega}-hg\Big)}{h}\\
	&=(\Delta u+k^2(1+\alpha(x))u+3k^2\beta(x)(u_{0}+\lambda u)^2u, \dfrac{ \partial u }{ \partial \nu }\Big\vert_{\partial\Omega}-g)\in B_{3},
\end{aligned}$$\\
\noindent
$l$ is holomorphic on $\mathbb{C}$. Therefore, $F$ is holomorphic.

Then, prove that $\partial_{u}F(0,0)$: $B_{2}\rightarrow B_{3}$ is a linear isomorphism.

$$\partial_{u}F(0,0):B_{2}\rightarrow B_{3}, $$
$$\partial_{u}F(0,0)v=(\Delta{v}+k^2(1+\alpha(x))v,\dfrac{ \partial v }{ \partial \nu }\Big\vert_{\partial\Omega}). $$
Because 0 is not a Neumann eigenvalue of $\Delta+k^2(1+\alpha(x))$, $\partial_{u}F(0,0)v=0$ if and only if $v=0$. Therefore $\partial_{u}F(0,0)$ is injection. From \cite[Lemma 2.2]{Mikko1}, it is easy to know that $\partial_{u}F(0,0)$ is surjection. Moreover, by the definition of $\partial_{u}F(0,0)$ that it is linear. In conclusion, $\partial_{u}F(0,0)$: $B_{2}\rightarrow B_{3}$ is a linear isomorphism.

By the implicit function theorem, ${\exists}$ $\delta>0$ and a unique holomorphic map $S$: $B_{\delta}(\partial\Omega)\rightarrow B_{2}$ s.t. for all $g\in B_{\delta}(\partial\Omega)$,
\begin{center}
	$F(g,S(g))=F(0,0),$ $S(0)=0.$
\end{center}

Let $u=S(g)$, because $S$ is holomorphic map on $B_{\delta}(\partial\Omega)$, which means $S$ about $g\in B_{\delta}(\partial\Omega)$ is holomorphic, and $S$ is Lipschitz continuous. Then
\begin{center}
 $\lVert u \rVert_{C^{2,\gamma}(\overline\Omega)}\leq C 	\lVert g \rVert_{C^{1,\gamma}(\partial\Omega)}$.
\end{center}
\end{proof}

\section{Uniqueness of the linear inverse problem}\label{sec3}

Consider the uniqueness of the following linear inverse problem which is the linear setting of (\ref{problem}).
\begin{equation}\label{xxfwt}
	\begin{cases}
		\Delta u+k^2(1+\alpha(x))u= 0 &in \ \Omega,\\
		\dfrac{ \partial u }{ \partial \nu }=g &on \ \partial\Omega,
	\end{cases}
\end{equation} 	
where $\Omega\subset\mathbb{R}^n$, $n\geq2$, is a connected bounded open domain with $C^{\infty}$ boundary, $\nu$ is the unit outward normal to $\partial \Omega$, the wavenumber $k$ is real, $\alpha(x)$$\in C^\gamma (\overline\Omega)$, $g(x)\in C^{1,\gamma}(\partial\Omega)$, $0<\alpha<1$, $\alpha(x)$, $g(x)$ are real value functions and $0$ is not the Neumann eigenvalue of $\Delta+k^2(1+\alpha(x))$.

Define the Neumann to Dirichlet map:
\begin{center}
	$N_{\alpha}:C^{1,\gamma}(\partial\Omega)\rightarrow C^{2,\gamma}(\partial\Omega)$,
\end{center}
\begin{center}
	$g\longmapsto u_{g}\vert_{\partial\Omega}$.
\end{center}

\begin{prop}\label{mt1}
Let $n\geq2$, $\alpha_{1}(x)$, $\alpha_{2}(x)\in C^\gamma (\overline\Omega)$, for all $g_{1}(x)$, $g_{2}(x)\in C^{1,\gamma}(\partial\Omega)$, we have
$$((N_{\alpha_{1}}-N_{\alpha_{2}})g_{1},g_{2})_{L^2(\partial\Omega)}=\int_{\Omega} -k^2(\alpha_{1}-\alpha_{2})u_{1}\overline{u_{2}} dx, $$
where $u_{1}$, $u_{2}$ is the solution of (\ref{xxfwt}) about $\alpha_{1}(x)$, $g_{1}(x)$ and $\alpha_{2}(x)$, $g_{2}(x)$.
\end{prop}

\begin{proof}
According to the definition of the inner product of $L^2$ and Green's formula, there are
$$\begin{aligned}
	(N_{\alpha}g_{1}, g_{2})_{L^2(\partial\Omega)}&=\int_{\partial\Omega}u_{1}\cdot\partial_{\nu}\overline{u_{2}}dS\\
	&=\int_{\Omega}u_{1}\cdot\Delta\overline{u_{2}}+\nabla u_{1}\cdot \nabla \overline{u_{2}}dx\\
	&=\int_{\Omega}\nabla u_{1}\cdot \nabla \overline{u_{2}}-k^2(1+\alpha(x))\overline{u_{2}}\cdot u_{1}dx\\
	&=\int_{\Omega}\nabla u_{1}\cdot \nabla \overline{u_{2}}+\overline{u_{2}}\cdot\Delta u_{1}dx\\
	&=\int_{\partial\Omega}\partial_{\nu}u_{1}\cdot \overline{u_{2}}dS\\
	&=(g_{1},N_{\alpha}g_{2})_{L^2(\partial\Omega)}.
\end{aligned}$$
\noindent It can be seen that
$$(N_{\alpha_{1}}g_{1},g_{2})_{L^2(\partial\Omega)}=\int_{\Omega}\nabla u_{1}\cdot \nabla \overline{u_{2}}-k^2(1+\alpha_{1}(x))\overline{u_{2}}\cdot u_{1}dx, $$
$$\begin{aligned}
	(N_{\alpha_{2}}g_{1}, g_{2})_{L^2(\partial\Omega)}&=(g_{1},N_{\alpha_{2}}g_{2})_{L^2(\partial\Omega)}\\
	&=\int_{\Omega}\nabla u_{1}\cdot \nabla \overline{u_{2}}-k^2(1+\alpha_{2}(x))\overline{u_{2}}\cdot u_{1}dx,
\end{aligned}$$
so we have $$\begin{aligned}
	(N_{\alpha_{1}}g_{1},g_{2})_{L^2(\partial\Omega)}-(N_{\alpha_{2}}g_{1},g_{2})_{L^2(\partial\Omega)}&=((N_{\alpha_{1}}-N_{\alpha_{2}})g_{1},g_{2})_{L^2(\partial\Omega)}\\
	&=\int_{\Omega} -k^2(\alpha_{1}-\alpha_{2})u_{1}\overline{u_{2}} dx.
\end{aligned}$$
\end{proof}



Next, we will prove the following uniqueness theorem for linear model the (\ref{xxfwt}).

\begin{thm}\label{xxfwtdl}
Let $\Omega\subset\mathbb{R}^n$, be a connected bounded open domain with $C^{\infty}$ boundary. \vspace{1mm}For $n\geq3$, $\alpha_{1}(x)$, $\alpha_{2}(x)\in C^{\alpha}(\overline\Omega)$, if $N_{\alpha_{1}}=N_{\alpha_{2}}$, then $\alpha_{1}(x)=\alpha_{2}(x)$ in $\Omega$. 
\vspace{1mm}For $n=2$, assume $\alpha_{1}(x)$, $\alpha_{2}(x)\in W^{1,p}(\Omega)$, $p>2$, if $N_{\alpha_{1}}=N_{\alpha_{2}}$, then $\alpha_{1}(x)=\alpha_{2}(x)$ in $\Omega$.
\end{thm}

For the case $n=2$, the uniqueness of the linear model is a direct consequence of \cite[Theorem 1]{2D2015}. Therefore, we will only provide the proof for the three-dimensional case here. Now, to ensure the completeness of the article, we first follow the proof process in \cite{Mikko2} to prove Lemma 3.3, and then obtain Proposition 3.4 and Proposition 3.5. Finally, the proof of Theorem \ref{xxfwtdl} for $n\geq3$ is given.

\begin{lem}\label{CGO}
Let  $\tau\in \mathbb{R}$, $\rho=\tau(e+i\varsigma)\in \mathbb{C}^{n}$, here $e$, $\varsigma\in\mathbb{R}^{n}$ are unit vectors, and $\alpha(x)\in C^{\alpha}(\overline{\Omega})$, there exist $C$, $\tau_{0}>0$ \vspace{1mm}such that for any $\rho$ that satisfies $e\cdot\varsigma=0$ and $\left|\tau\right|\geq\frac{\sqrt{2}}{2}\tau_{0}$, the equation $\Delta u+k^2(1+\alpha(x))u= 0$ has a solution $u\in H^{1}(\Omega)$, and
\begin{equation}
	u=e^{Re(\rho)\cdot x}(e^{iIm(\rho)\cdot x}+r),
\end{equation}
where $r\in L^{2}(\Omega)$ satisfies
\begin{equation}
	\lVert r \rVert_{L^{2}(\Omega)}\leq\frac{C}{\left|\rho\right|}.
\end{equation}
\end{lem}

Without loss of generality, by selecting an appropriate coordinate system, we take  $x=(x_{1},x^{'})\in\mathbb{R}^{n}$, $e=(1,0,\cdots,0)\in\mathbb{R}^{n}$ and then $u$ be the following form
\begin{equation}\label{CGOxs}
	u=e^{\tau x_{1}}(a(x^{'})+r)\in H^{1}(\Omega),
\end{equation}
where
\begin{equation}
	a(x^{'})=e^{i\tau\varsigma\cdot x}.
\end{equation}

Next, we define the operator
$$P_{\tau}: v\longmapsto e^{-\tau x_{1}}\Delta(e^{\tau x_{1}}v). $$
Then
$$\begin{aligned}
	\Delta u&=\Delta (e^{\tau x_{1}}(a(x^{'})+r))\\
	&=e^{\tau x_{1}}\cdot P_{\tau}(a(x^{'})+r),
\end{aligned}$$
and
$$\begin{aligned}
	\Delta u+k^2(1+\alpha(x))u&=e^{\tau x_{1}}\cdot P_{\tau}(a(x^{'})+r)+k^2(1+\alpha(x))u\\
	&=e^{\tau x_{1}}(P_{\tau}+k^2(1+\alpha(x)))(a(x^{'})+r).
\end{aligned}$$

Let $f=-(P_{\tau}+k^2(1+\alpha(x)))a(x^{'})$, then
 (\ref{CGOxs}) is the solution of $\Delta u+k^2(1+\alpha(x))u= 0\hspace*{0.1cm}, in\hspace*{0.1cm} \Omega$ if and only if $r$ is the solution of $(P_{\tau}+k^2(1+\alpha(x)))r=f\hspace*{0.1cm}, in\hspace*{0.1cm} \Omega$. In that way, to prove Lemma \ref{CGO} is equivalent to prove Proposition \ref{djCGO} as follows.

\begin{prop}\label{djCGO}
Let $\alpha(x)\in C^{\alpha}(\overline{\Omega})$, there are $C$, $\tau_{0}>0$, such that $\left|\tau\right|\geq\tau_{0}$. For any $f\in L^{2}(\Omega)$,
\begin{equation}
	(P_{\tau}+k^2(1+\alpha(x)))u=f\hspace*{0.1cm} in\hspace*{0.1cm} \Omega
\end{equation}
has a solution $u\in L^{2}(\Omega)$ satisfies
\begin{equation}
	\lVert u \rVert_{L^{2}(\Omega)}\leq\frac{C}{\left|\tau\right|}.
\end{equation}
\end{prop}

\begin{prop}\label{CGOgj}
Let $\alpha(x)\in C^{\alpha}(\overline{\Omega})$, there are $C$, $\tau_{0}>0$, such that $\left|\tau\right|\geq\tau_{0}. $ And we have the following estimate
\begin{equation}
	\lVert u \rVert_{L^{2}(\Omega)}\leq\frac{C}{\left|\tau\right|}\lVert (P_{-\tau}+k^2(1+\alpha(x)))u \rVert_{L^{2}(\Omega)},\hspace*{0.2cm}u\in C^{\infty}_{c}(\Omega).
\end{equation}
\end{prop}
\begin{proof}
Based on the definition of $P_{\tau}$, it follows that
	$$\begin{aligned}
		P_{\tau}v&=e^{-\tau x_{1}}\Delta(e^{\tau x_{1}}v)\\
		&=(\Delta+\tau^2+2\tau\partial_{1})v.
	\end{aligned}$$
Note $P_{-\tau}v=e^{\tau x_{1}}\Delta(e^{-\tau x_{1}}v)$, then
	$$\begin{aligned}
		P_{-\tau}&=\Delta+\tau^2-2\tau\partial_{1}\\
		&=A+iB,
	\end{aligned}$$
where $A=\Delta+\tau^2, B=2i\tau\partial_{1}$.
	
For any $u\in C^{\infty}_{c}(\Omega)$, there are
	$$\begin{aligned}
		\lVert P_{-\tau}u \rVert^2_{L^2(\Omega)}
		&=(P_{-\tau}u, P_{-\tau}u)_{L^2(\Omega)}\\
		&=((A+iB)u, (A+iB)u)_{L^2(\Omega)}\\
		&=(Au, Au)_{L^2(\Omega)}+(Bu, Bu)_{L^2(\Omega)}+i(Bu, Au)_{L^2(\Omega)}-i(Au, Bu)_{L^2(\Omega)}\\
		&=\lVert Au \rVert^2_{L^2(\Omega)}+\lVert Bu \rVert^2_{L^2(\Omega)}+i(ABu, u)_{L^2(\Omega)}-i(BAu, u)_{L^2(\Omega)}\\
		&=\lVert Au \rVert^2_{L^2(\Omega)}+\lVert Bu \rVert^2_{L^2(\Omega)}+(i(AB-BA)u, u)_{L^2(\Omega)}.
	\end{aligned}$$
It is easy to see $AB=BA$, therefore
	$$\begin{aligned}
		\lVert P_{-\tau}u \rVert_{L^2(\Omega)}
		&\geq \lVert Bu \rVert_{L^2(\Omega)}\\
		&=2\left| \tau \right| \lVert \partial_{1}u \rVert_{L^2(\Omega)}.
	\end{aligned}$$

As $u\in C^{\infty}_{c}(\Omega)$, there are $a$, $b\in\mathbb{R}$ such that $supp(u)\subset\Omega\subset\{a<x_{1}<b\}$. For any $x\in\Omega$, we have
$$\int_{a}^{x_{1}}\partial_{1}u(\theta,x')d\theta=u(x_{1},x'). $$
By H\"{o}der inequality, there are
$$\begin{aligned}
	\lvert u(x_{1},x')\lvert
	&=\lvert\int_{a}^{x_{1}}\partial_{1} u(\theta,x') \, d\theta\lvert \\[2mm]
	&\leq \int_{a}^{x_{1}}\lvert \partial_{1} u(\theta,x')\lvert \, d\theta\\[2mm]
	&\leq (b-a)^{\tfrac{1}{2}}\lVert \partial_{1} u(x)\rVert_{L^2(\Omega)}.
\end{aligned}$$
Square both sides of the inequality simultaneously and integrate on $\Omega$, then
\begin{equation*}
	\int_{\Omega} \lvert u(x_{1},x')\lvert^{2} \,dx
	\leq (b-a)\lvert \Omega\rvert \lVert \partial_{1} u(x)\rVert_{L^2(\Omega)}^2.
\end{equation*}
Therefore, there are $C_{0}>0$, such that
	$$\lVert u \rVert_{L^{2}(\Omega)}\leq C_{0}\lVert \partial_{1}u \rVert_{L^{2}(\Omega)}, $$
then
	$$\begin{aligned}
		\lVert u \rVert_{L^{2}(\Omega)}
		&\leq \frac{C_{0}}{2\left| \tau\right|}\lVert P_{-\tau}u \rVert_{L^2(\Omega)}\\
		&\leq\frac{C_{0}}{2\left| \tau\right|}(\lVert P_{-\tau}u+k^2(1+\alpha(x))u \rVert_{L^2(\Omega)}+ \lVert-k^2(1+\alpha(x))u\rVert_{L^2(\Omega)}).
	\end{aligned}$$
Suppose $\left|\tau\right|\geq 2C_{0}\lVert k^2(1+\alpha(x)))\rVert_{L^{\infty}(\Omega)}$, there are $C>0$ such that
	$$\lVert u \rVert_{L^{2}(\Omega)}\leq\frac{C}{\left|\tau\right|}\lVert (P_{-\tau}+k^2(1+\alpha(x)))u \rVert_{L^{2}(\Omega)}. $$
\end{proof}

Now we finish the proof of Proposition \ref{djCGO}.

\textbf{The proof of Proposition \ref{djCGO}:}
Define $X=(P_{-\tau}+k^2(1+\alpha(x)))C^{\infty}_{c}(\Omega)$, $X$ is a subspace of $L^2(\Omega)$. Take $f\in L^{2}(\Omega)$ arbitrarily and define a linear functional
\begin{center}
	$\phi:X \rightarrow \mathbb{C}$,
\end{center}
$$\phi((P_{-\tau}+k^2(1+\alpha(x)))\psi)= \displaystyle\int_{\Omega}\overline{f}\psi dx. $$
As indicated by Proposition \ref{CGOgj}, $(P_{-\tau}+k^2(1+\alpha(x)))\psi=0$ if and only if $\psi =0$, then $(P_{-\tau}+k^2(1+\alpha(x)))$ is injection, for each $(P_{-\tau}+k^2(1+\alpha(x)))\psi\in X$ has a unique $\psi\in C^{\infty}_{c}(\Omega)$ corresponding to it. Therefore, $\phi$ is well-defined.

From Proposition \ref{CGOgj} again and the Schwarz inequality, we have
$$\begin{aligned}
	\left| \phi((P_{-\tau}+k^2(1+\alpha(x)))\psi)\right|
	&=\left|\displaystyle\int_{\Omega}f\overline{\psi} dx\right|\\
	&=\left| (\psi,f)_{L^{2}(\Omega)} \right|\\
	&\leq\lVert f\rVert_{L^2(\Omega)}\dot\lVert\psi\rVert_{L^2(\Omega)} \\
	&\leq\frac{C}{\left| \tau\right|}\lVert f\rVert_{L^2(\Omega)}
	\lVert (P_{-\tau}+k^2(1+\alpha(x)))\psi \rVert_{L^2(\Omega)}.
\end{aligned}$$
It can be seen that $\phi$ is a bounded linear functional on $X$, so $\phi$ is also a continuous linear functional on $X$.

Since $X\subset L^{2}(\Omega)$, $\phi$ is a continuous and linear functional on $X$, by Hahn-Banach theorem,
$\phi$ can be extended to a continuous and linear functional $\Phi$ in $L^2(\Omega)$, and
$$\lVert\Phi\rVert_{L^2(\Omega)}\leq \frac{C}{|\tau|}\lVert f \rVert_{L^2(\Omega)}. $$
By Riesz Representation Theorem, there exists a unique $u$ $\in L^2(\Omega)$ such that for any $\omega\in L^{2}(\Omega)$, there are
$$\begin{aligned}
	\Phi(\omega)
	&=(\omega,u)_{L^{2}(\Omega)}\\
	&=\int_{\Omega}\omega\overline{u} dx,
\end{aligned}$$
and
$$\lVert u\rVert_{L^2(\Omega)}=\lVert \Phi\rVert_{L^2(\Omega)}\leq \frac{C}{|\tau|}\lVert f\rVert_{L^2(\Omega)}. $$

As $\psi \in C^{\infty}_{c}(\Omega)$, by Green formula
$$\begin{aligned}
	(\psi,(P_{\tau}+k^2(1+\alpha(x)))u)_{L^2(\Omega)}
	&=((P_{-\tau}+k^2(1+\alpha(x)))\psi,u)_{L^2(\Omega)}\\
	&=\Phi(P_{-\tau}+k^2(1+\alpha(x))\psi)\\
	&=\phi(P_{-\tau}+k^2(1+\alpha(x))\psi)\\
	&=\int_{\Omega}\overline{f}\psi dx\\
	&=(\psi,f)_{L^2(\Omega)}.
\end{aligned}$$
Therefore, $(P_{\tau}+k^2(1+\alpha(x)))u=f$ has a solution $u\in L^2(\Omega)$ and
$$\lVert u\rVert_{L^2(\Omega)}\leq \frac{C}{|\tau|}. $$
\hfill\textbf{$\square$}

So far, we have completed the proof of Lemma \ref{CGO} by proving Proposition \ref{djCGO} and Proposition \ref{CGOgj}.

\begin{lem}\label{Runge}
Suppose that $k(x)\in C^{\gamma}(\overline{\Omega})$. Then for any solution $V\in H^1(\Omega)$ to
$\Delta V+k(x)V=0\ \mathrm{in}\ \Omega$, there exist a solution $U\in C^{2, \gamma}(\overline{\Omega})$ to
$\Delta U+k(x)U=0\ \mathrm{in}\ \Omega,$
such that  for any $\eta>0$
\begin{align*}
\parallel V-U\parallel_{L^2(\Omega)}<\eta.
\end{align*}
\end{lem}
\begin{proof}
Define
 $$X=\{ U\in C^{2, \gamma}(\overline{\Omega})|\ U\ is\ a\ solution\ to\ \Delta U+k(x)U=0\ \mathrm{in}\ \Omega\},$$
and $$Y=\{ V\in H^{1}({\Omega})|\ V\ is\ a \ solution\ to\ \Delta V+k(x)V=0\ \mathrm{in}\ \Omega\}.$$
It is easy to see $X$ and $Y$ is not empty and $X\subseteq Y$.

To prove Lemma 3.6, as long as to prove that $X$ is dense in $Y$. Therefore, we only need to prove that for a fixed $y \in L^{2}(\Omega)$, if
$$(y,U)_{L^2(\Omega)}=0,\hspace{5mm} \forall U\in X, $$
then
$$(y,V)_{L^2(\Omega)}=0,\hspace{5mm} \forall V\in Y. $$

Now we consider the following equation
\begin{align}\label{Runge eq}
	\begin{cases}
		\displaystyle \Delta W+k(x)W=y\ &in \, \Omega,\\
		\displaystyle \frac{\partial W}{\partial \nu}=0\ &on \, \partial\Omega,
	\end{cases}
\end{align}
by \cite[Lemma 2.2]{Mikko1}, we konw that (\ref{Runge eq}) have a uniqueness solution $W\in H^1(\Omega)$.

For any $U\in X$, we find the fixed $y\in L^2(\Omega)$ such that
$$\begin{aligned}
	0=\int_{\Omega}y\overline{U} dx
	&=\int_{\Omega}(\Delta W+k(x)W)\overline{U} dx\\
	&=\int_{\Omega}\overline{U}\Delta W dx-\int_{\Omega}W\Delta \overline{U} dx\\
	&=-\int_{\partial\Omega}W \partial_{\nu}\overline{U} dS
\end{aligned}$$
	
Since for any $U\in X$ the above equation holds, we can take $U\in X$ such that $\partial_{\nu}\overline{U}\neq 0$, then we will see that $W\equiv 0 \hspace{2mm} on \, \partial\Omega$.

Therefore, for any $V\in Y$, we have
$$\begin{aligned}
	\int_{\Omega}y\overline{V} dx
	&=\int_{\Omega}(\Delta W+k(x)W)\overline{V} dx\\
	&=\int_{\Omega}\overline{V}\Delta W dx-\int_{\Omega}W\Delta \overline{V} dx\\
	&=-\int_{\partial\Omega}W \partial_{\nu}\overline{V} dS\\
	&=0.
\end{aligned}$$
\end{proof}

With the above theoretical foundation, we can now complete the proof of Theorem \ref{xxfwtdl} for $n\geq3$.

\textbf{The proof of Theorem \ref{xxfwtdl} for $\bf{n\geq3}$:}

Firstly, fix a vector $\xi\in\mathbb{R}^n$ and use assumption $n\geq3$ to find some unit vectors $\alpha$, $\gamma\in\mathbb{R}^n$ such that
$\{$$\sigma$, $\vartheta$, $\xi$ $\}$ is an orthogonal set.

Next, for $\tau\geq|\xi|$, define two complex vectors
$$\rho_{1}=\tau\sigma+i[\xi+\sqrt{\tau^2-|\xi|^2}\vartheta], $$
$$\rho_{2}=-\tau\sigma+i[\xi-\sqrt{\tau^2-|\xi|^2}\vartheta], $$
where $|Re(\rho_{j})|=|Im(\rho_{j})|=\tau$, $Re(\rho_{j})\perp Im(\rho_{j})$.

By Lemma \ref{CGO}, if $\tau\geq\tau_{0}$ is large enough,
then
$$\Delta u+k^2(1+\alpha_{j}(x))u= 0\hspace*{0.1cm},\ \text{in}\ \Omega,\hspace*{0.1cm} j=1,2, $$
have the solution
$$u_{j}=e^{Re(\rho_{j})\cdot x}(e^{iIm(\rho_{j})\cdot x}+r_{j}),$$
when $\tau\rightarrow \infty$,  $\lVert r_{j}\rVert_{L^2(\Omega)}\rightarrow 0$.

By Lemma \ref{Runge}, there exist $\{\hat{u}_{1n}\}$, $\{\hat{u}_{2n}\}\in C^{2,\alpha}(\overline \Omega)$, which satisfy
$$\Delta \hat{u}_{jn}+k^2(1+\alpha_{j}(x))\hat{u}_{jn}= 0\hspace*{0.1cm},\ \text{in}\ \Omega,\hspace*{0.1cm} j=1,2. $$
such that $$\lim_{n\rightarrow\infty}\lVert\hat{u}_{1n}-u_{1}\rVert_{L^2(\Omega)}=0, $$
$$\lim_{n\rightarrow\infty}\lVert\hat{u}_{2n}-u_{2}\rVert_{L^2(\Omega)}=0. $$
Because of $N_{\alpha_{1}}=N_{\alpha_{2}}$, by proposition \ref{mt1}  $$\int_{{\Omega}}-k^2(\alpha_{1}(x)-\alpha_{2}(x))\hat{u}_{1n}\hat{u}_{2n} dx=0, $$
we have
$$\begin{aligned}
	\lim_{n\rightarrow\infty}\int_{{\Omega}}-k^2(\alpha_{1}(x)-\alpha_{2}(x))\hat{u}_{1n}\hat{u}_{2n} dx
	&=\int_{{\Omega}}-k^2(\alpha_{1}(x)-\alpha_{2}(x)){u}_{1} {u}_{2} dx\\
	&=0.
\end{aligned}$$
Then
$$\int_{{\Omega}}-k^2(\alpha_{1}(x)-\alpha_{2}(x))e^{2ix\cdot\xi} dx$$
$$=\int_{{\Omega}}-k^2(\alpha_{1}(x)-\alpha_{2}(x))(e^{iIm(\rho_{1})\cdot x}r_{2}+e^{iIm(\rho_{2})\cdot x}r_{1}+r_{1} r_{2}) dx, \hspace*{0.1cm}\forall \xi \in \mathbb{R}^n. $$
Since $|e^{iIm(\rho_{j})\cdot x}|=1$, $\tau\rightarrow \infty$, $r_{1}\rightarrow 0$, $r_{2}\rightarrow 0$. When $\tau\rightarrow \infty$,  $$\int_{{\Omega}}-k^2(\alpha_{1}(x)-\alpha_{2}(x))e^{2ix\cdot\xi} dx=0,\ \forall\ \xi \in \mathbb{R}^n.$$
Therefore, $$\alpha_{1}(x)=\alpha_{2}(x). $$
\hfill\textbf{$\square$}

\section{Multiparameter determination in the semilinear He-lmholtz equation}\label{sec4}

From Section \ref{sec3}, we can obtain the following density results respectively.
\begin{prop}\label{cmx>2}
Let $\Omega\in \mathbb{R}^n$, $n\geq3$, be a connected bounded domain with $C^{\infty}$ boundary. If exist $f\in C^{\gamma}(\overline{\Omega})$, $0<\gamma<1$, such that
	$$\int_{{\Omega}} f u_{1} u_{2} \, dx=0 $$
for any $u_{1}$, $u_{2}\in C^{2,\gamma}(\overline{\Omega})$ satisfying $\Delta u+k^2(1+\alpha(x))u= 0\hspace*{0.1cm} in \,\Omega, $ then $f=0$ in $\Omega$.
\end{prop}

\begin{prop}\label{cmx2}
Let $\Omega\in \mathbb{R}^2$, be a connected bounded domain with $C^{\infty}$ boundary. If exist $f\in W^{1,p}(\Omega)$, $p>2$ such that
$$\int_{{\Omega}} f u_{1} u_{2} \, dx=0 $$
for any $u_{1}$, $u_{2}\in C^{2,\gamma}(\overline{\Omega})$ satisfying $\Delta u+k^2(1+\alpha(x))u= 0\hspace*{0.1cm} in \,\Omega$, then $f=0$ in $\Omega$. 	
\end{prop}

\begin{rem}
Proposition \ref{cmx>2} can be easily obtained through Lemma \ref{CGO} and similar to the proof of Theorem \ref{xxfwtdl} regarding $n\geq3$. 
Proposition \ref{cmx2} can be obtained by the proof process of Theorem 1 in \cite{2D2015} and using Runge type approximation similar to Lemma \ref{Runge}.
\end{rem}

Next, we use higher order linearization methods to prove Theorem \ref{dl1} and Theorem \ref{dl2}. Since the proof of Theorem \ref{dl1} is similar to that in \cite{llls2020}, we have present it in the appendix.
For Theorem \ref{dl2} ($n=2$ case), it can be similarly proved by using Theorem \ref{xxfwtdl} and Propsition \ref{cmx2}.

\section{Numerical reconstruction}\label{sec5}
In this section, we will discuss the reconstruction algorithm for multi-coefficients in the semilinear Helmholtz equation and give some numerical examples.

\subsection{Reconstruction algorithm}
We consider $\Omega$ to be the rectangular region $[0,1] \times [0,1]$, as shown in Figure~\ref{fig:demo5.1}. The domain $\Omega$ is uniformly partitioned with step size $h = 1/N$, so that $x_n = (n-1)h$, $n=1, \ldots, N+1$, and $y_m = (m-1)h$, $m=1, \ldots, N+1$; see Figure~\ref{fig:demo5.2}.
\begin{figure}[H]
	\centering
	\begin{minipage}[t]{0.35\textwidth}
		\centering
		\includegraphics[width=\linewidth]{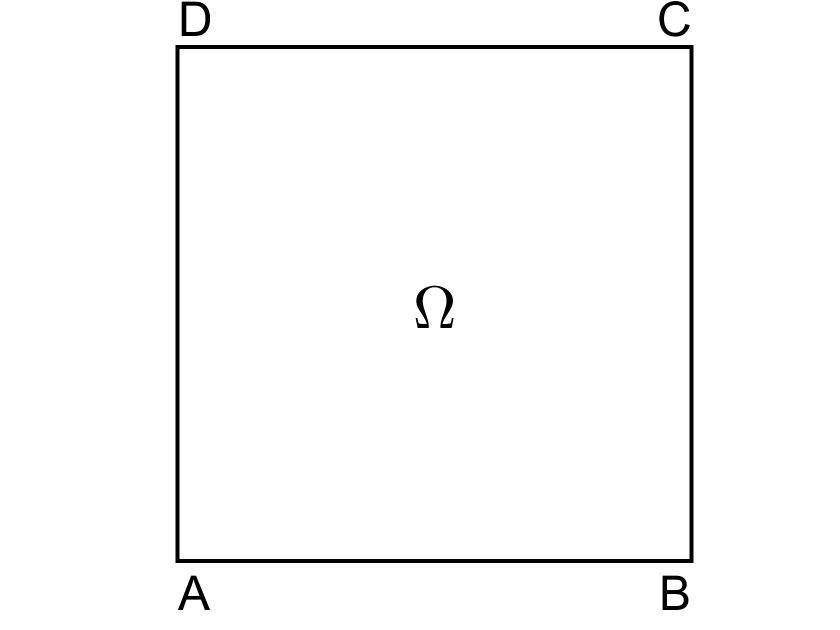}
		\caption{region ABCD}
		\label{fig:demo5.1}
	\end{minipage}
	\begin{minipage}[t]{0.35\textwidth}
		\centering
		\includegraphics[width=\linewidth]{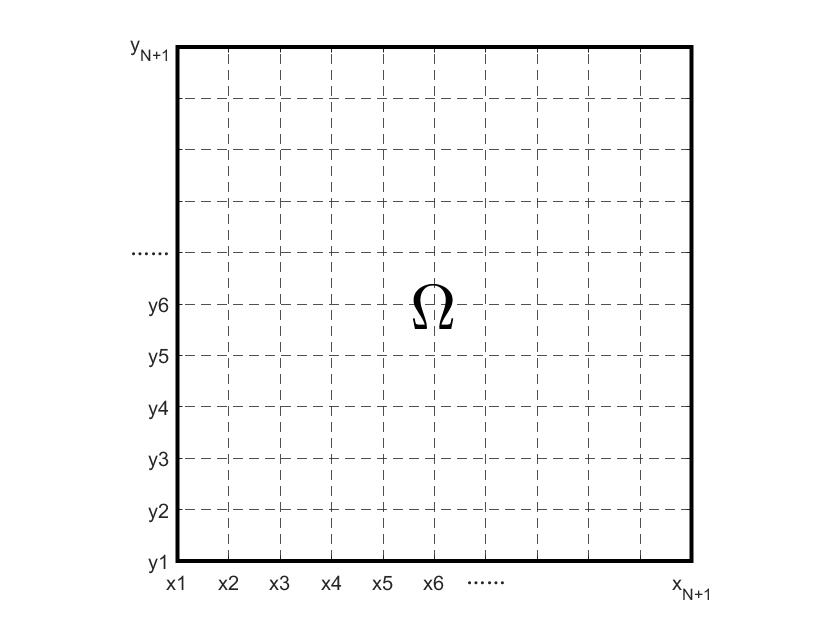}
		\caption{Divided ABCD}
		\label{fig:demo5.2}
	\end{minipage}
\end{figure}
Next, we will use the following five-point finite difference scheme within $\Omega$.
\begin{equation}\label{lsgs1}
	\dfrac{(u_{n-1,m}+u_{n,m+1}+u_{n+1,m}+u_{n,m-1}-4u_{n,m})}{h^2}+k^2(1+\alpha_{n,m})u_{n,m}+k^2\beta_{n,m}u^3_{n,m}=0,
\end{equation}
where $n=2,3,\text{…}, N, m=2,3,\text{…}, N.$ And we separately adopt the following discretization scheme for the boundaries AB, BC, CD, DA.
\begin{align}
	(\dfrac{u_{n+1,1}-u_{n,1}}{h}, \dfrac{u_{n,2}-u_{n,1}}{h})\cdot(0,-1)=g_{n,1} \quad\quad\quad\hspace{2mm} on \hspace{1mm}AB,
	\\[2mm]
	(\dfrac{u_{N+1,m}-u_{N,m}}{h}, \dfrac{u_{N+1,m+1}-u_{N+1,m}}{h})\cdot(1,0)=g_{N+1,m}\quad\quad\hspace{0.1mm} on \hspace{1mm}BC,
	\\[2mm]
	(\dfrac{u_{n+1,N+1}-u_{n,N+1}}{h}, \dfrac{u_{n,N+1}-u_{n,N}}{h})\cdot(0,1)=g_{n,N+1}\quad\quad\hspace{0.6mm} on \hspace{1mm}CD,
	\\[2mm]
	(\dfrac{u_{2,m}-u_{1,m}}{h}, \dfrac{u_{1,m+1}-u_{1,m}}{h})\cdot(-1,0)=g_{1,m}\quad\quad\quad\hspace{1mm}  on \hspace{1mm}DA.
\end{align}
Note that for the four vertices\vspace{2mm} A, B, C, D,  each vertex satisfies discretization scheme with two boundaries. We combine the weights of $\dfrac{1}{2}$ with\vspace{2mm} the discrete format of two boundaries to obtain the discretization scheme of the following vertices.
\begin{align}
	u_{11}=\dfrac{1}{2}(u_{1,2}+u_{2,1}+2g_{1,1}h) \quad\quad\quad\hspace{2.05cm} at \hspace{1mm}A,   \\[2mm]
	u_{N+1,1}=\dfrac{1}{2}(u_{N,1}+u_{N+1,2}+2g_{N+1,1}h)
	\quad\quad\hspace{1.25cm} at \hspace{1mm}B,
	\\[2mm]
	u_{N+1,N+1}=\dfrac{1}{2}(u_{N,N+1}+u_{N+1,N}+2g_{N+1,N+1}h)
	\quad\quad\hspace{0.65mm} at \hspace{1mm}C,
	\\[2mm]
	u_{1,N+1}=\dfrac{1}{2}(u_{1,N}+u_{2,N+1}+2g_{1,N+1}h)
	\quad\quad\quad\hspace{0.785cm}  at \hspace{1mm}D.
\end{align}


After providing the discretization scheme, we select different $m$, $n$ from (\ref{lsgs1}) yield different equations, which can be combined into a vector $F$, as shown in (\ref{F})-(\ref{Fi}). Finally, we use the quasi Newton method for iteration to calculate the numerical solution of the forward problem, as shown in Table \ref{table:demo20}.
\vspace{2mm}
\begin{equation}\label{F}
	F=(F_{1}, F_{2}, F_{3}, \text{…}, F_{(N-1)\cdot(N-1)})_{1\times(N-1)\cdot(N-1)},
\end{equation}
\vspace{3mm}
\begin{equation}\label{Fi}
	F_{i}=\dfrac{(u_{n-1,m}+u_{n,m+1}+u_{n+1,m}+u_{n,m-1}-4u_{n,m})}{h^2}+k^2(1+\alpha_{n,m})u_{n,m}+k^2\beta_{n,m}u^3_{n,m},
\end{equation}
where $i=(n-2)(N-1)+(m-1),\hspace{1mm} n=2,3,\text{…}, N,\hspace{1mm} m=2,3,\text{…}, N.$
\begin{table}[H]
	\centering
	\caption{Quasi Newton method}
	\label{table:demo20}
	\begin{tabular}{llllllll}
		\toprule
		Algorithm 1    \\
		\midrule
		1: \quad Give $u^{(0)}$, $M$, termination error $\delta$, let $t=0$;   \\
    	2: \quad \textbf{for} $t=0$ \textbf{to} $M$ \textbf{do}  \\
		3: \quad Calculate the Jacobian matrix of F:  $JF(u^{(t)})$;\\
		4: \quad $u^{(t+1)}=u^{(t)}-F(u^{(t)})/JF(u^{(t)})$;\\
		5: \quad If $\lVert u^{(t+1)}-u^{(t)}\rVert_{2}/\lVert u^{(t)}\rVert_{2} <\delta$, end;\\
		6: \quad else if, $u^{(t+1)}=u^{(t)}$;\\
		7: \quad \textbf{end for}\\
		\bottomrule
	\end{tabular}
\end{table}
After solving the numerical solution of the forward problem, we obtain its boundary measurement data $u|_{\partial\Omega}$, and using Bayesian method \cite{Dashti2015} to solve the inverse problem by boundary data.

Let $\mathbf{X}$ be a separable Hilbert space equipped with the Borel $\sigma$-algebra, and let $\mathcal{G} : \mathbf{X} \to \mathbb{R}^n$ be a measurable function, referred to as the forward operator, representing the relationship between the model parameters and the data, as in \eqref{problem}. The inverse problem consists of determining the unknown parameter $c \in \mathbf{X}$ from measurement data $y \in \mathbb{R}^n$, which is typically given by
\begin{align}\label{ForMod}
y = \mathcal{G}(c) + \eta,
\end{align}
where the noise $\eta$ is assumed to be a zero-mean Gaussian random variable with covariance matrix $\Gamma$.
In the Bayesian framework, both the model parameter and the measurement data are treated as random variables. From \eqref{ForMod}, the negative log-likelihood is given by
$$
\Phi(c):=\frac{1}{2}|\Gamma^{-\frac{1}{2}}(\mathcal{G}(c)-y)|^2.
$$
Combining the prior measure $\mu_0$ with density $\pi_0$ and Bayes' theorem, the posterior density, up to a normalizing constant, is
\begin{align}\label{PosDen}
\pi(c) = \exp\bigl(-\Phi(c)\bigr)\pi_0(c).
\end{align}
To characterize the posterior, we employ the pCN algorithm \cite{Dashti2015} for sampling, as summarized in Table \ref{table:demo19}.

\begin{figure}[htbp]
 \centering
 \includegraphics[width=0.9\textwidth,height=0.45\textwidth]{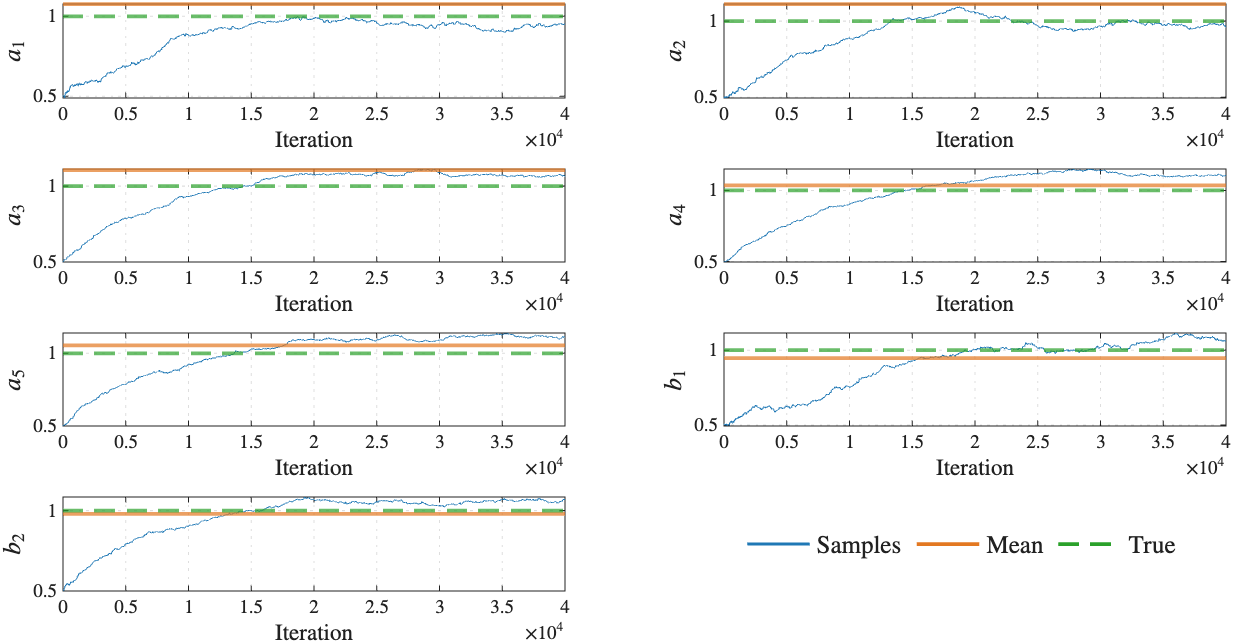}
\caption{Sample traces of all coefficients in Example \hyperref[exa1]{1}.
}
\label{fig:Traceploy}
\end{figure}


\begin{table}[H]
	\centering
	\caption{pCN method}
	\label{table:demo19}
	\begin{tabular}{llllllll}
		\toprule
		Algorithm 2   \\
		\midrule
		1: \quad Give $c^{(0)}$, $\gamma$, $M$, let $t=0$;   \\
		2: \quad \textbf{for} $t=0$ \textbf{to} $M$ \textbf{do}  \\
		3: \quad
		Extract sample $s^{(t)}=\sqrt{1-\gamma^2}c^{(t)}+\gamma \xi^{(t)}$, where $\xi^{(t)}\sim \mathcal{N}(0,C_{0})$;\\
    	4: \quad
    	Calculate the acceptance probability $\rho(c^{(t)},s^{(t)})=\min\{1,\exp(\Phi(c^{(t)})-\Phi(s^{(t)}))\}$;\\
		5: \quad Generate a uniform random number $\theta\sim U[0,1]$;\\
		6: \quad If $\theta\leq \rho(c^{(t)},s^{(t)})$, then take $c^{(t+1)}=s^{(t)}$;\\
		7: \quad else if, $c^{(t+1)}=c^{(t)}$;\\
		8: \quad \textbf{end for}\\
		\bottomrule
	\end{tabular}
\end{table}

\begin{figure}[htbp]
 \centering
 \includegraphics[width=1\textwidth,height=0.5\textwidth]{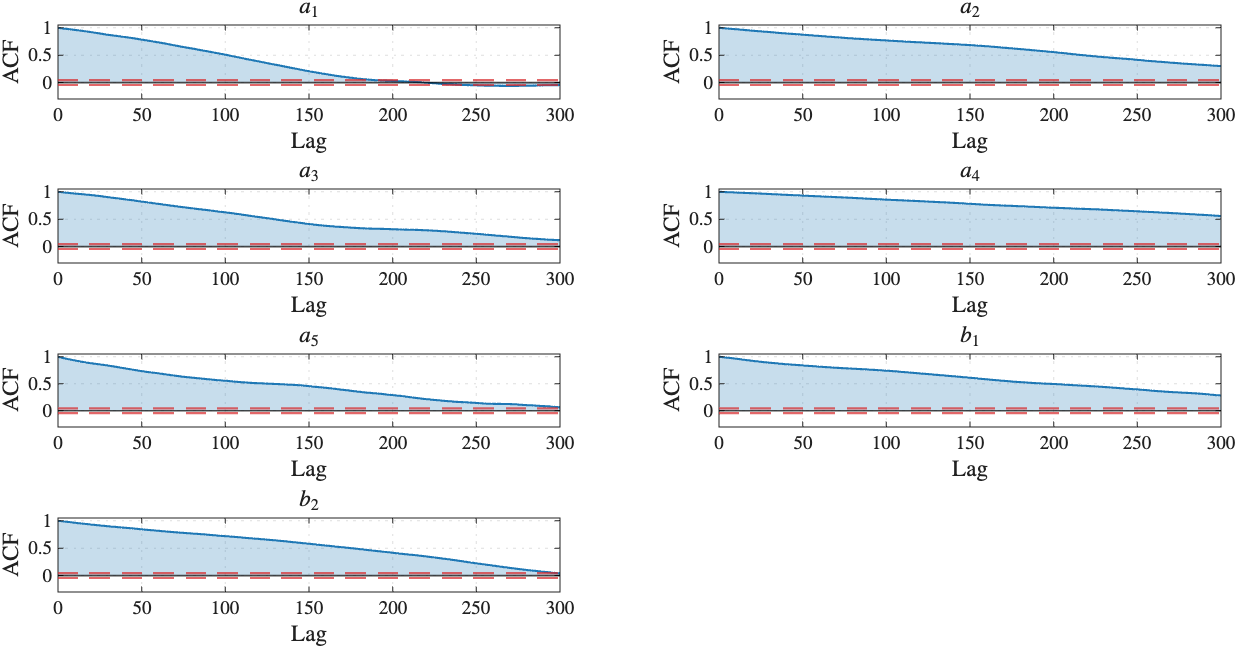}
\caption{Sample autocorrelation functions of all coefficients in Example \hyperref[exa1]{1}.
}
\label{fig:ACFploy}
\end{figure}

\begin{figure}[htbp]
  \centering
  \subfloat[True, $\alpha$]
  {
 \includegraphics[width=0.27\textwidth,height=0.22\textwidth]{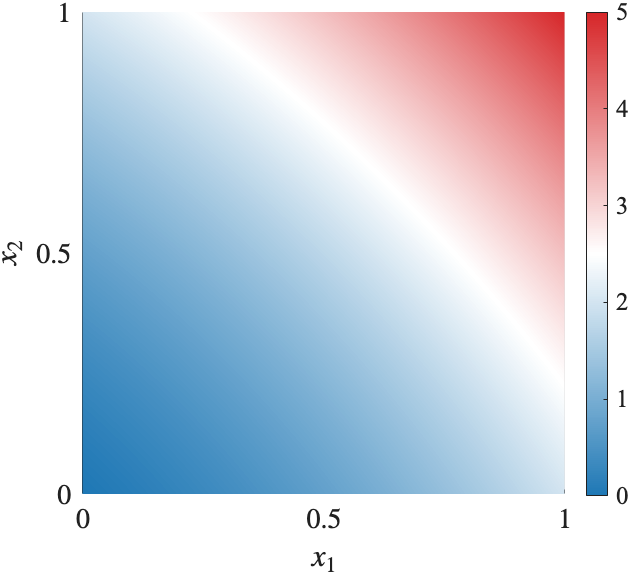}
  }
\subfloat[Reconstruction, $\alpha$]
  {
    \includegraphics[width=0.27\textwidth,height=0.22\textwidth]{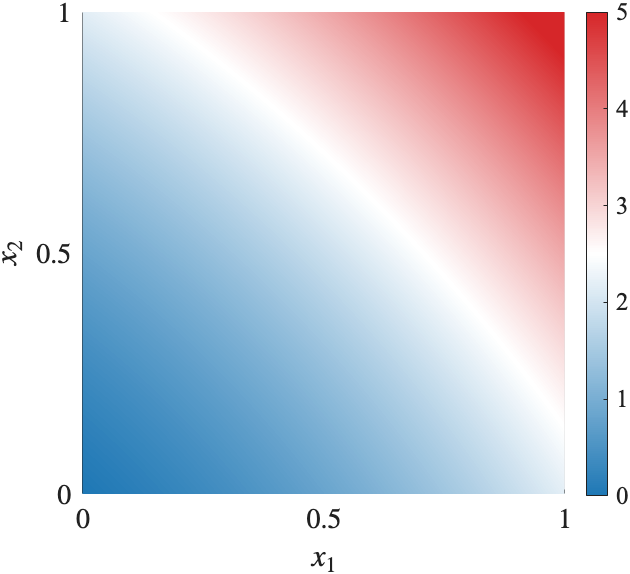}
  }
   \subfloat[$\sigma$, $\alpha$]
  {
  \includegraphics[width=0.28\textwidth,height=0.23\textwidth]{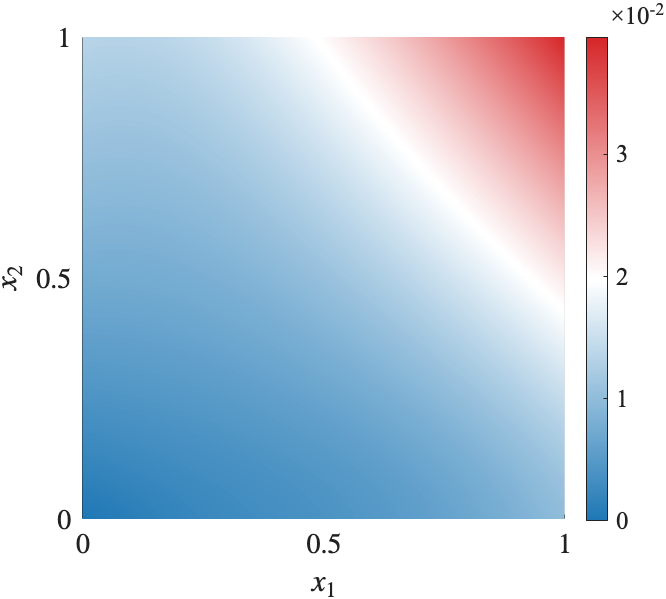}
  }\\
    \subfloat[True, $\beta$]
  {
 \includegraphics[width=0.27\textwidth,height=0.22\textwidth]{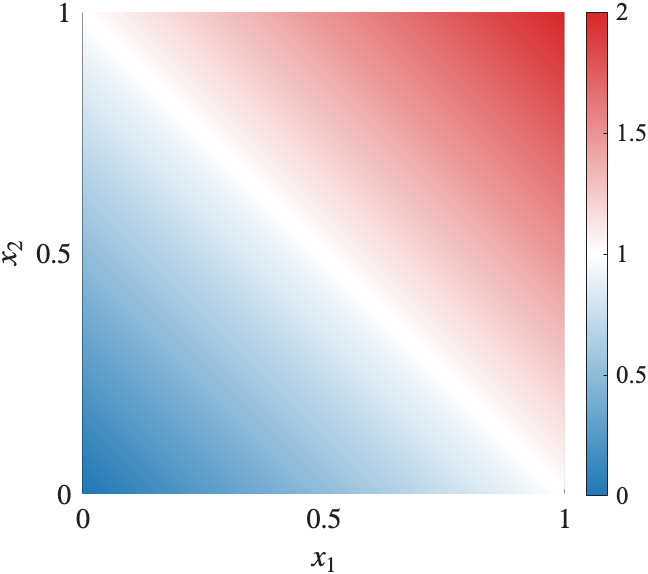}
  }
\subfloat[Reconstruction, $\beta$]
  {
    \includegraphics[width=0.27\textwidth,height=0.22\textwidth]{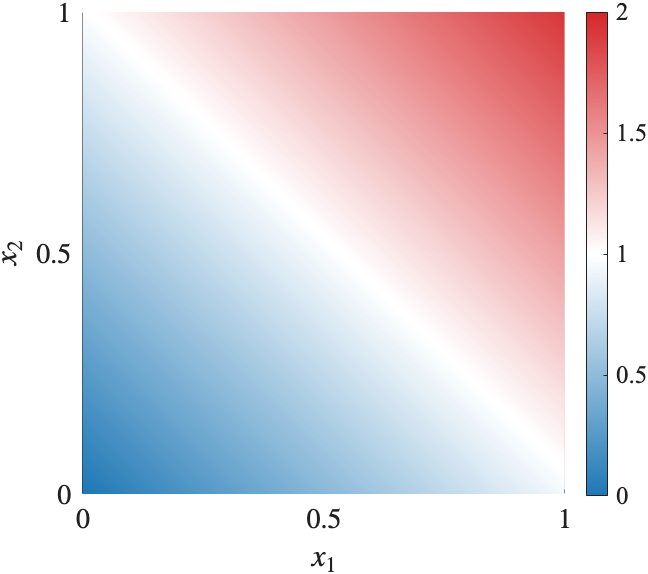}
  }
   \subfloat[$\sigma$, $\beta$]
  {
  \includegraphics[width=0.28\textwidth,height=0.23\textwidth]{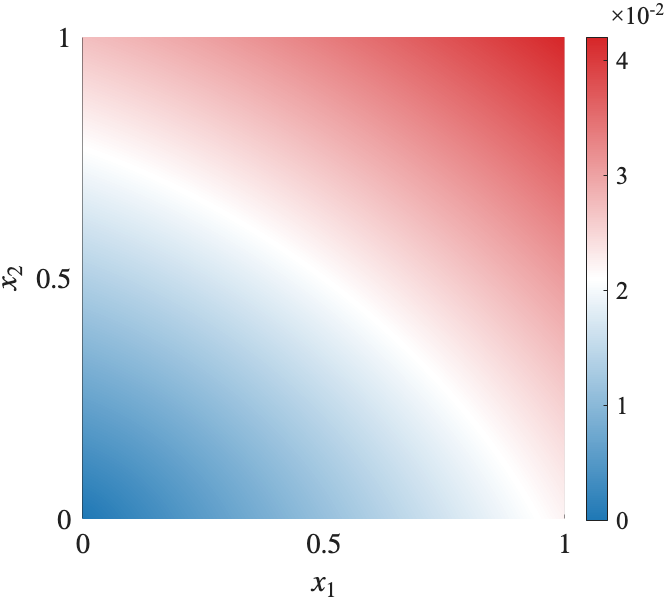}
  }
  
\caption{Sample mean and standard deviation ($\sigma$) for reconstructing $\alpha$ and $\beta$ in Example \hyperref[exa1]{1}}
    \label{fig:ex1}
\end{figure}

\subsection{Numerical examples}
In following numerical examples, we use the finite difference method with  Quasi Newton iteration, i.e., Algorithm \ref{table:demo20} to solve the forward problem, and then reconstruct the $\alpha$ and $\beta$ by Bayesian method.
\subsubsection{Example 1}\label{exa1}
Let $k=1, N=40$, and $ g(x)=x_{1}-x_{2}$, where $x=(x_1,x_2)$.
The basis for space $L^2(\Omega)$ is truncated as
$$
\{x_1,x_2,x_1^2,x_1x_2,x_2^2,\dots,x_1^{\mathbb{K}},x_1^{\mathbb{K}-1}x_2,\dots,x_2^{\mathbb{K}}\},
$$
where $\mathbb{K}$ is a positive integer.
For computational simplicity, in this example we reconstruct the $(\alpha,\beta)$ in this polynomial basis.
The exact $\alpha$ and $\beta$ are set
$$
\alpha(x)=a_1x_{1}+a_2x_{2}+a_3x_{1}^2+a_4x_{1}x_{2}+a_5x_{2}^2,\quad \beta(x)=b_1x_{1}+b_2x_{2},
$$
with $a_i=b_j=1$ for $i=1,\dots,5,j=1,2$, respectively.
Thus, we reconstruct the coefficients $\theta=[a_1,a_2,\dots, a_5,b_1,b_2]$.
We set a Gaussian prior $\mu_0$ with mean $0.5\times\text{ones}(7,1)$ and covariance matrix $\text{diag}([0.045,0.045,0.03,0.03,0.035,0.05,0.035]^2)$.
Measurement data are generated via $u=\mathcal{G}(\alpha,\beta)+\eta$ where $\mathcal{G}$ is the forward model and $\eta$ represents the Gaussian noise with a standard deviation taken by $0.1\%$ (i.e., noise level) of the maximum norm of the model output $u|_{\partial\Omega}$.
Notice that, in order to avoid `inverse crimes', we compute the data by on a finer grid than used in the inversion.

\begin{figure}[htbp]
 \centering
 \includegraphics[width=0.8\textwidth,height=0.4\textwidth]{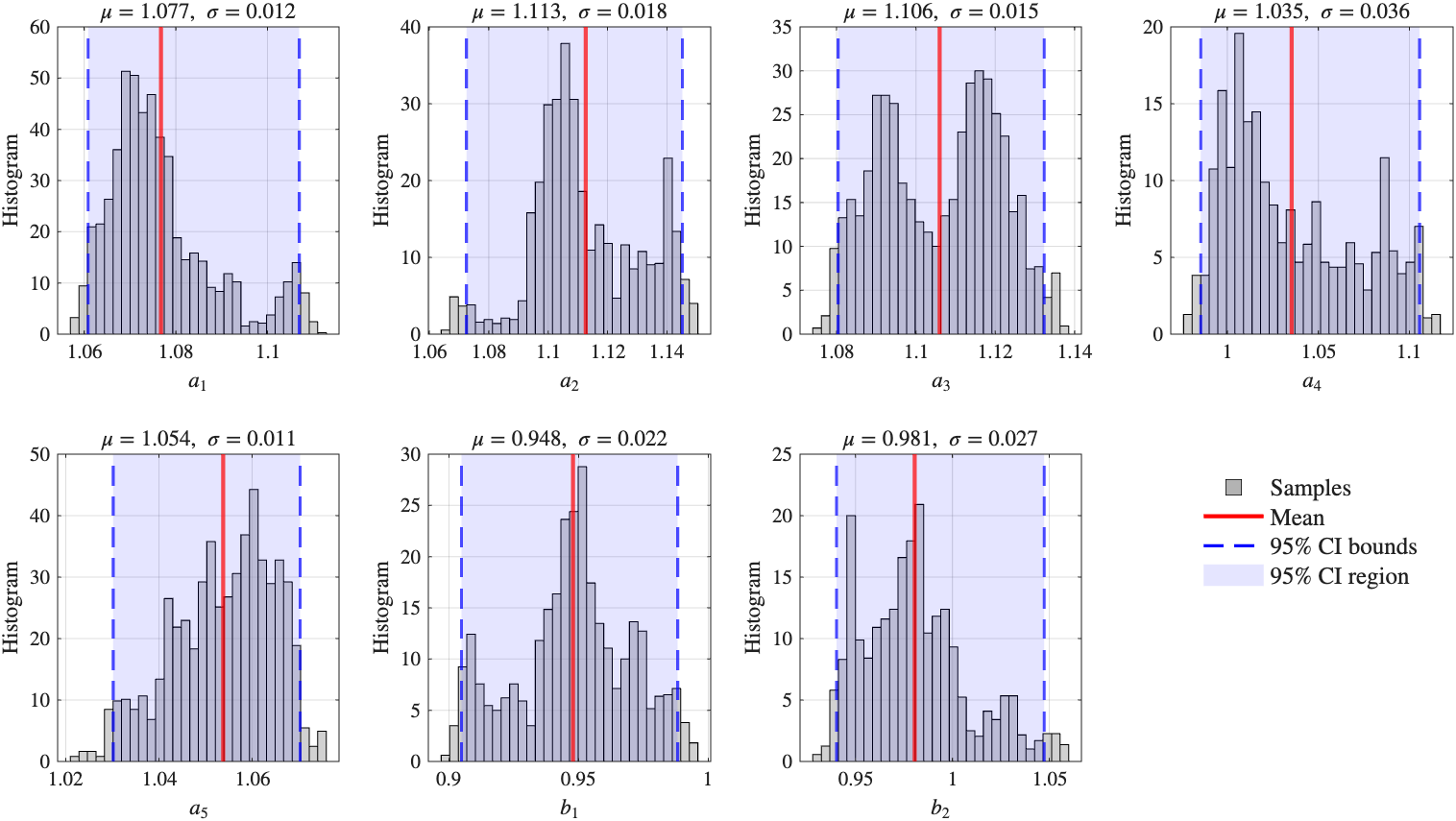}
\caption{Samples histogram in Example \hyperref[exa1]{1} and $\mu$ represents the samples mean.
}
\label{fig:Histploy}
\end{figure}

We run the pCN algorithm \ref{table:demo19}, generating $4 \times 10^4$ samples and discarding the first $2 \times 10^4$ as burn-in.
Figure \ref{fig:Traceploy} shows the trace plots of all coefficients.
All chains appear to stabilize around their true values after the burn-in period, with no visible trends or drift, indicating convergence.
Figure \ref{fig:ACFploy} displays the autocorrelation functions (ACF) up to lag $300$; all components show a monotone decay toward zero, and correlations become negligible within a few hundred lags, confirming satisfactory mixing and a reasonable effective sample size. 
Figure \ref{fig:ex1} displays the sample mean and standard deviation ($\sigma$) of the reconstructed $\alpha$ and $\beta$.
The true are shown in Figure \ref{fig:ex1}a and \ref{fig:ex1}d.
It can be observed that the $\alpha$ and the $\beta$ are very close to the true, and both achieve a small sample standard deviation.
The marginal posterior distributions in Figure \ref{fig:Histploy} are multimodal; their means are close to the exact value $1$, and the true coefficients for $a_4$ and $b_2$ lie within the empirical 95\% credible intervals. 
The non-Gaussian shape suggests that the posterior is not well-approximated by a Gaussian distribution, a common feature in nonlinear inverse problems with Gaussian priors.
The correlation matrix in Figure \ref{fig:Corrploy} reveals that the cross-correlations between the 
$\alpha$ and $\beta$ coefficients are generally weak (most of the absolute values are less than $0.3$), indicating that these two sets of coefficients are approximately independent in the posterior distribution given the observational data. In a Bayesian sense, this suggests that  $\alpha$ and $\beta$ are decoupled. Meanwhile, the strong correlations observed within each coefficients set--such as the negative correlations between $a_1$ and $a_3$ and between $a_2$ and $a_5$--reflect trade-offs in their contributions to the forward model. 



\begin{figure}[htbp]
 \centering
 \includegraphics[width=0.5\textwidth,height=0.4\textwidth]{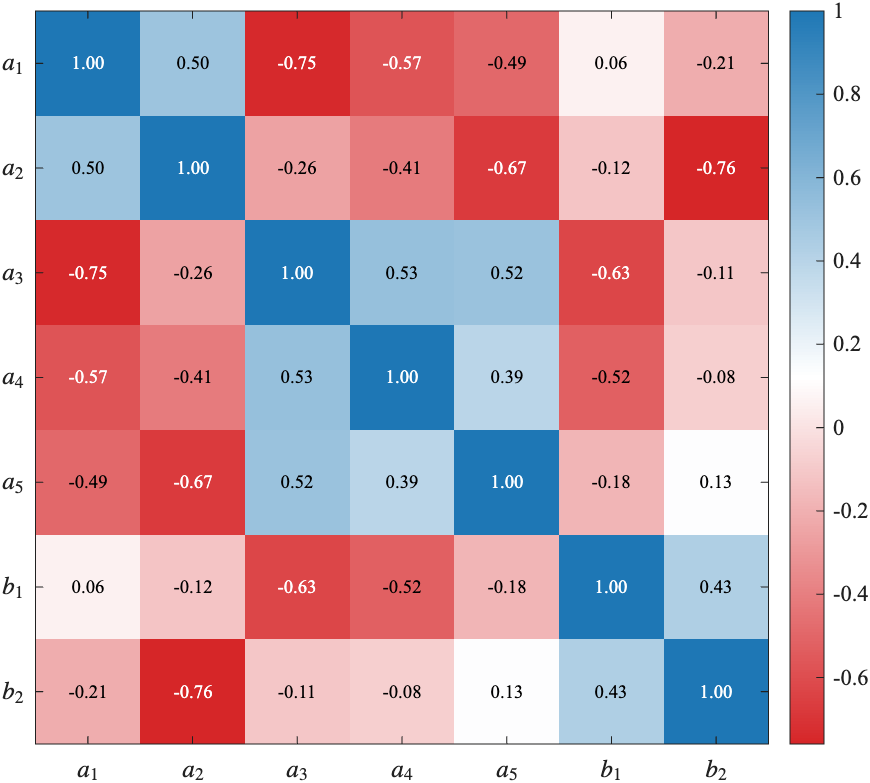}
\caption{Correlation coefficients matrix in Example \hyperref[exa1]{1}.
}
\label{fig:Corrploy}
\end{figure}

\begin{figure}[htbp]
 \centering
 \includegraphics[width=0.9\textwidth,height=0.45\textwidth]{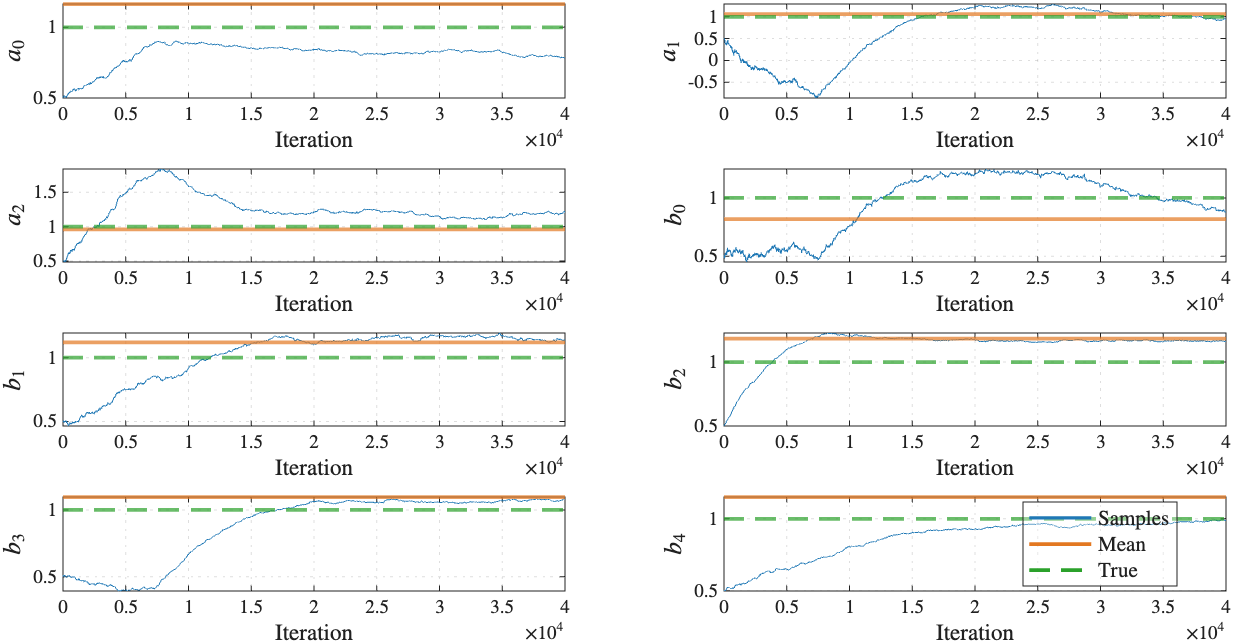}
\caption{Sample traces of all coefficients in Example \hyperref[exa2]{2}.
}
\label{fig:TraceTri}
\end{figure}

\begin{figure}[htbp]
 \centering
 \includegraphics[width=1\textwidth,height=0.5\textwidth]{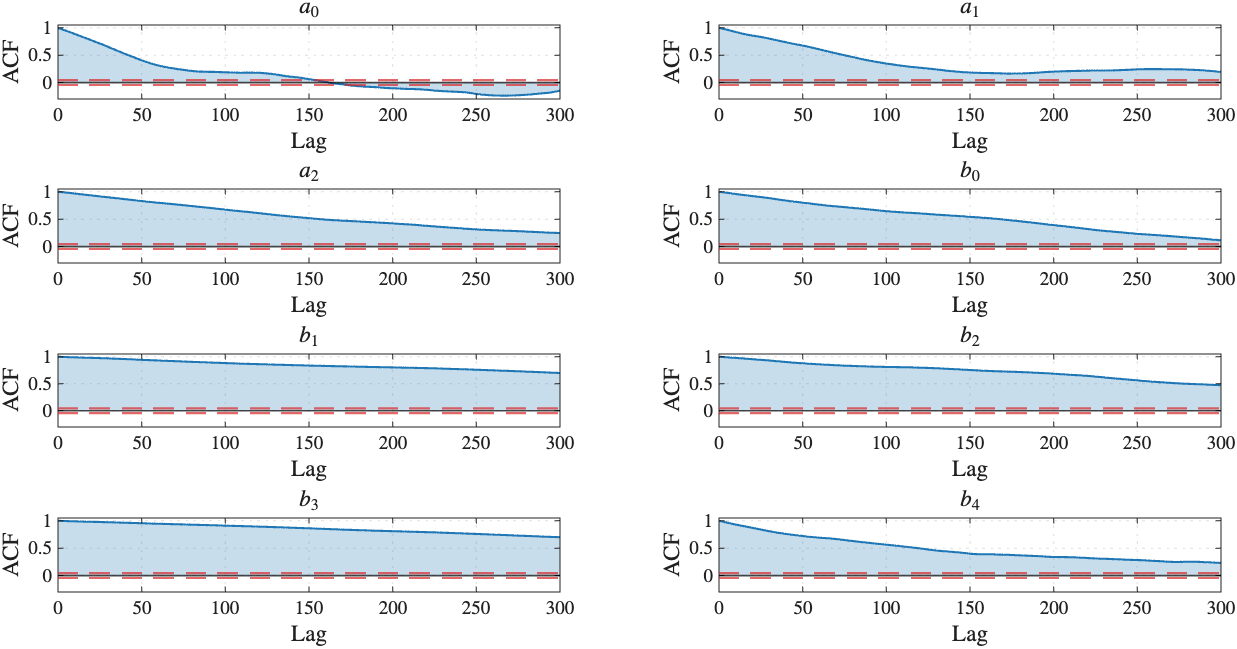}
\caption{Sample autocorrelation functions of all coefficients in Example \hyperref[exa2]{2}.
}
\label{fig:ACFTri}
\end{figure}

\subsubsection{Example 2}\label{exa2}
Let $x=(x_1,x_2)$. 
The basis functions are truncated as
\begin{align*}
&\bigl\{cos(\pi nx_1)cos(\pi mx_2)+cos(\pi nx_1)sin(\pi mx_2)
+sin(\pi nx_1)cos(\pi mx_2)\\
&+sin(\pi nx_1)sin(\pi mx_2)\bigr\}_{n,m=0}^{\mathbb{N},\mathbb{M}}
\end{align*}
where $\mathbb{N},\mathbb{M}$ are integers.
For computational simplicity, we reconstruct the coefficients $(\alpha, \beta)$ using this trigonometric basis. 
The exact forms of $\alpha$ and $\beta$ are given by
\begin{align*}
\alpha(x)&=a_0+a_1cos(\pi x_{2})+a_2sin(\pi x_{2}) \\
\beta(x)&=b_0+b_1cos(\pi x_{1})+b_2sin(\pi x_{1})+b_3cos(\pi x_2)+b_4sin(\pi x_2),
\end{align*}
with $a_i=b_j=1$ for $j=0,\dots,4,i=0,1,2$, respectively.
Thus, we reconstruct the coefficients $\theta=[a_0,a_1,a_2,b_0,\dots,b_4]$.
The Gaussian prior $\mu_0$ is chosen with mean $0.5\times\text{ones}(8,1)$ and covariance matrix $\operatorname{diag}([0.04, 0.28, 0.1, 0.12, 0.06, 0.03, 0.035, 0.03]^2)$. Measurement data are generated by $u = \mathcal{G}(\alpha, \beta) + \eta$, where $\mathcal{G}$ denotes the forward model and $\eta$ is Gaussian noise with standard deviation $0.1\%$ of the maximum norm of $u|_{\partial \Omega}$.
To avoid the so-called `inverse crimes', the measurement data are produced by solving the forward problem on a finer grid.

\begin{figure}[htbp]
  \centering
  \subfloat[True, $\alpha$]
  {
 \includegraphics[width=0.27\textwidth,height=0.22\textwidth]{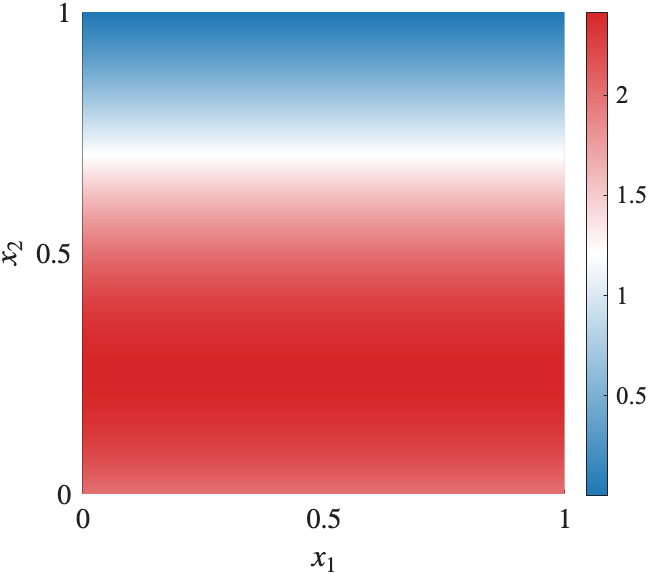}
  }
\subfloat[Reconstruction, $\alpha$]
  {
    \includegraphics[width=0.27\textwidth,height=0.22\textwidth]{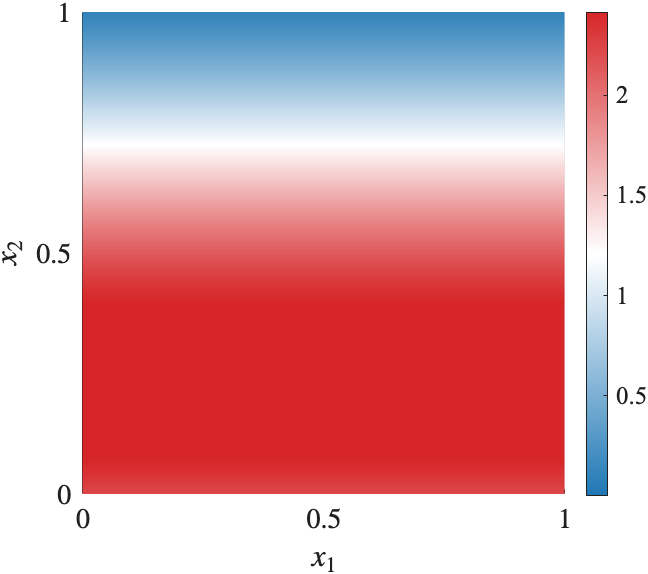}
  }
   \subfloat[$\sigma$, $\alpha$]
  {
  \includegraphics[width=0.28\textwidth,height=0.23\textwidth]{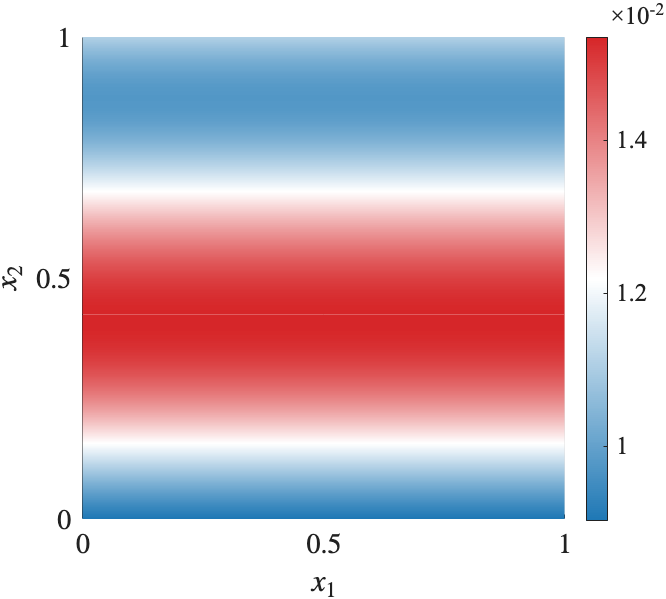}
  }\\
    \subfloat[True, $\beta$]
  {
 \includegraphics[width=0.27\textwidth,height=0.22\textwidth]{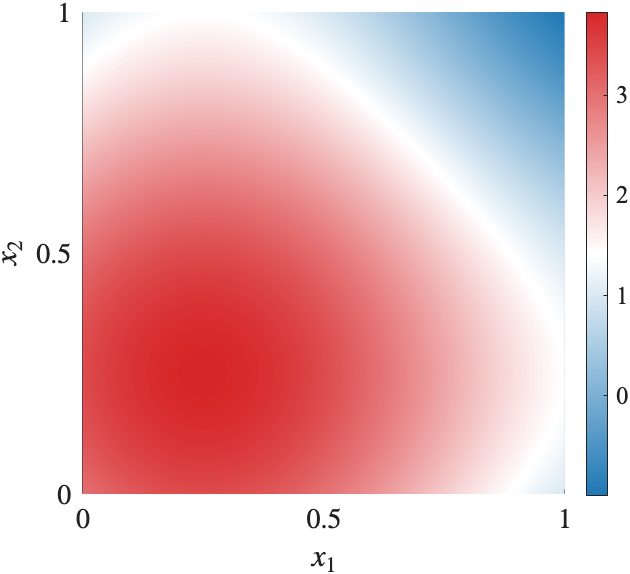}
  }
\subfloat[Reconstruction, $\beta$]
  {
    \includegraphics[width=0.27\textwidth,height=0.22\textwidth]{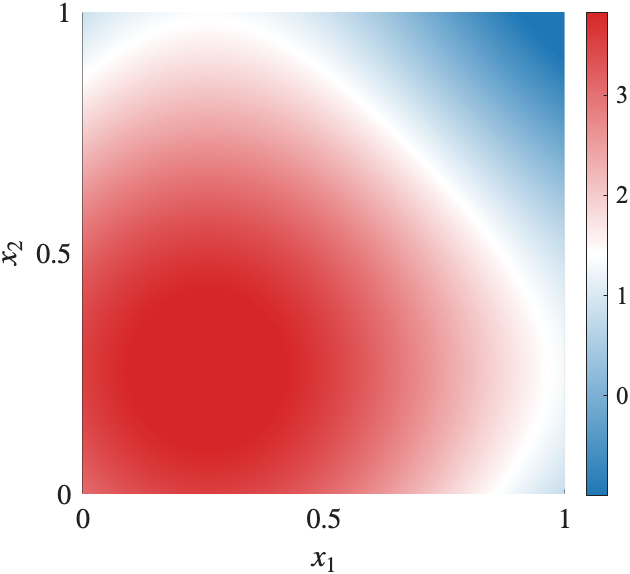}
  }
   \subfloat[$\sigma$, $\beta$]
  {
  \includegraphics[width=0.28\textwidth,height=0.23\textwidth]{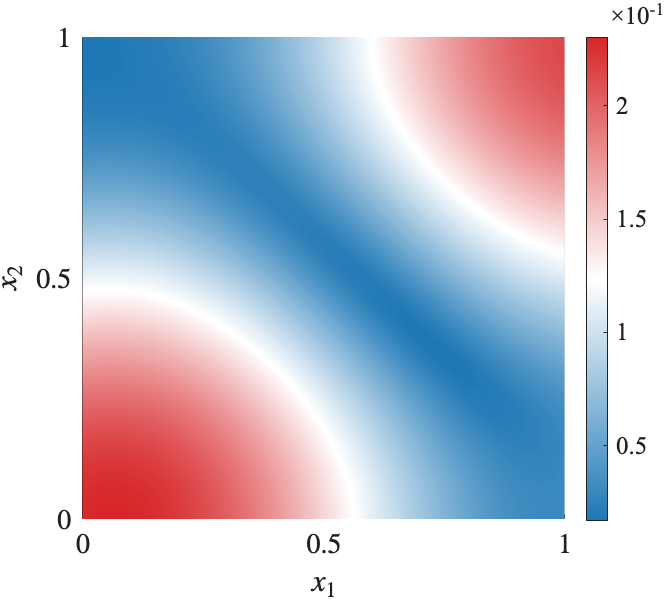}
  }

\caption{Sample mean and standard deviation ($\sigma$) for reconstructing $\alpha$ and $\beta$ in Example \hyperref[exa2]{2}}
    \label{fig:ex2}
\end{figure}

We run the pCN algorithm \ref{table:demo19} to generate $4\times10^4$ samples, discarding the first $2\times10^4$ iterations as burn-in. Figure~\ref{fig:TraceTri} presents the trace plots of all coefficients. After burn-in, each chain fluctuates stably around its true value without visible trends or drift, indicating convergence to stationarity. The autocorrelation functions in Figure~\ref{fig:ACFTri}, computed up to lag $300$, exhibit a monotone decay toward zero; correlations become negligible within a few hundred lags, suggesting satisfactory mixing and a reasonable effective sample size.
The sample means and standard deviations of the reconstructed $(\alpha,\beta)$ are shown in Figure~\ref{fig:ex2}, where the true functions are shown in Figure~\ref{fig:ex2}a and \ref{fig:ex2}d. The reconstructions closely match the ground truth, with small posterior standard deviations, demonstrating accurate recovery and limited uncertainty.
Figure~\ref{fig:HistTri} displays the marginal posterior distributions of all coefficients, together with posterior means and $95\%$ credible intervals. Several marginals are approximately symmetric and mildly multimodal. The posterior means are close to the true values $a_i=b_j=1$, confirming reliable identification of the ground-truth coefficients, although the posterior uncertainties vary substantially across parameters.
Finally, Figure~\ref{fig:CorrTri} shows the posterior correlation matrix of $\theta=[a_0,a_1,a_2,b_0,\dots,b_4]$. Pronounced cross-correlations between the $\alpha$- and $\beta$-coefficients are observed, with significant negative correlations (notably between $a_2$ and $b_1,b_3$). This indicates posterior coupling between $\alpha$ and $\beta$, reflecting the propagation of uncertainty across the two parameter sets through the forward model.

\begin{figure}[htbp]
 \centering
 \includegraphics[width=0.8\textwidth,height=0.4\textwidth]{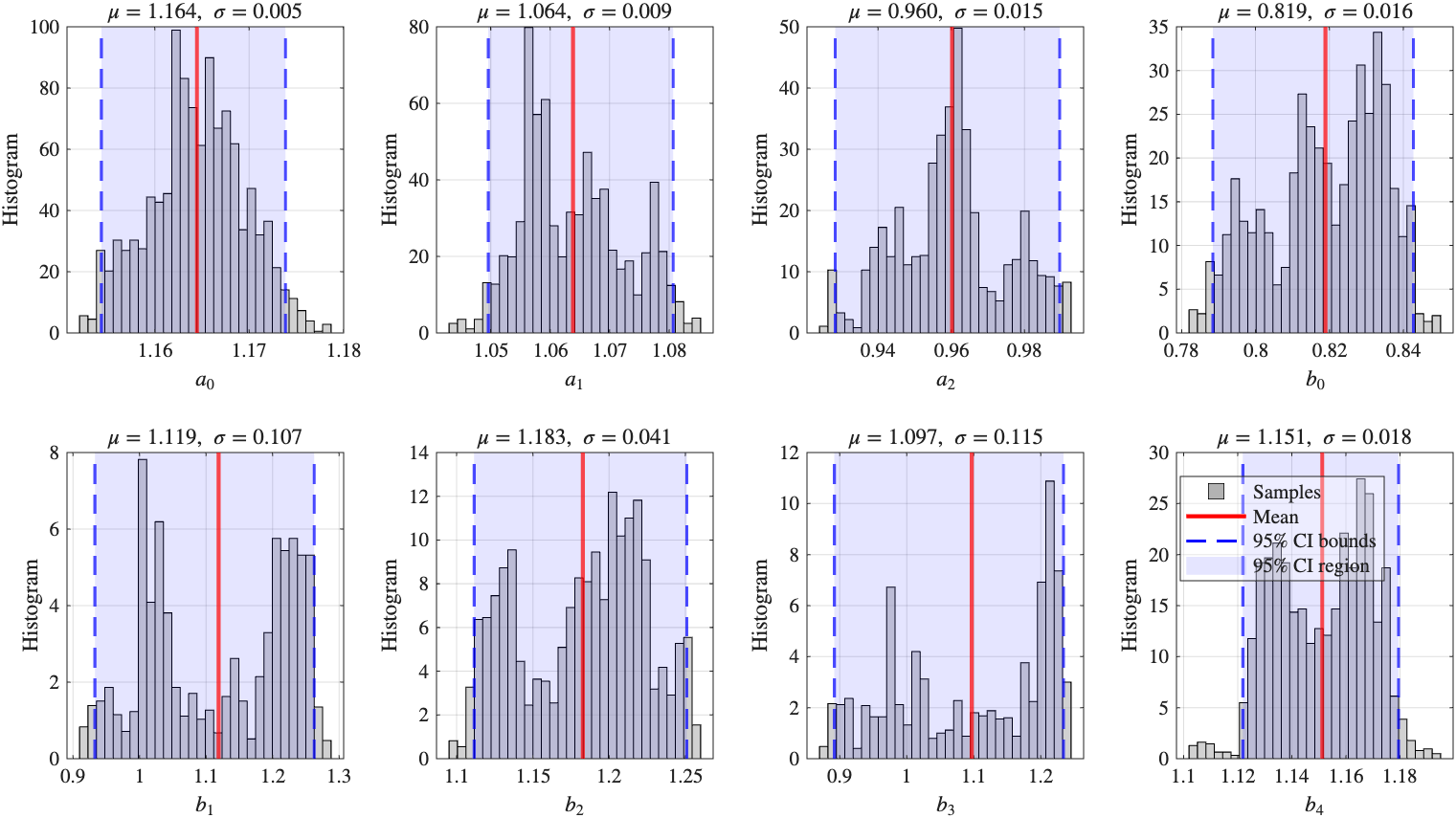}
\caption{Samples histogram in Example \hyperref[exa2]{2} and $\mu$ represents the samples mean.
}
\label{fig:HistTri}
\end{figure}

\section{Conclusion}\label{sec6}
In this work, we have established the unique recovery of both the linear and nonlinear coefficients $\alpha(x)$ and $\beta(x)$ in the semilinear Helmholtz equation from the Neumann-to-Dirichlet map. The analysis is carried out in two parts: for dimensions $n \geq 3$ under H\"older continuity assumptions, and for $n = 2$ under Sobolev regularity conditions. The proof strategy combines the higher-order linearization method with techniques from linear inverse problems, including the construction of CGO solutions and Runge-type approximation arguments.

The well-posedness of the forward problem is first established via the implicit function theorem, which allows us to define the NtD map and proceed with the linearization approach. The uniqueness results for the linearized problem are then extended to the fully nonlinear case through the density argument and higher-order linearization, demonstrating that the boundary data uniquely determine the interior coefficients.

In addition to the theoretical analysis, we develop a numerical reconstruction framework for recovering the coefficients $\alpha(x)$ and $\beta(x)$ from boundary data. 
The forward problem is discretized by a finite difference scheme combined with a quasi-Newton iteration, while the inverse problem is formulated within a Bayesian inference framework and solved via posterior sampling. 
The numerical experiments demonstrate that the proposed reconstruction method can effectively recover both coefficients and provide uncertainty quantification for the reconstruction results.

Future work may include extending these results to more general nonlinearities, considering partial data settings, or addressing the case of less regular coefficients. The methods developed here may also be applicable to other types of nonlinear wave equations and coupled systems.

\begin{figure}[htbp]
 \centering
 \includegraphics[width=0.5\textwidth,height=0.4\textwidth]{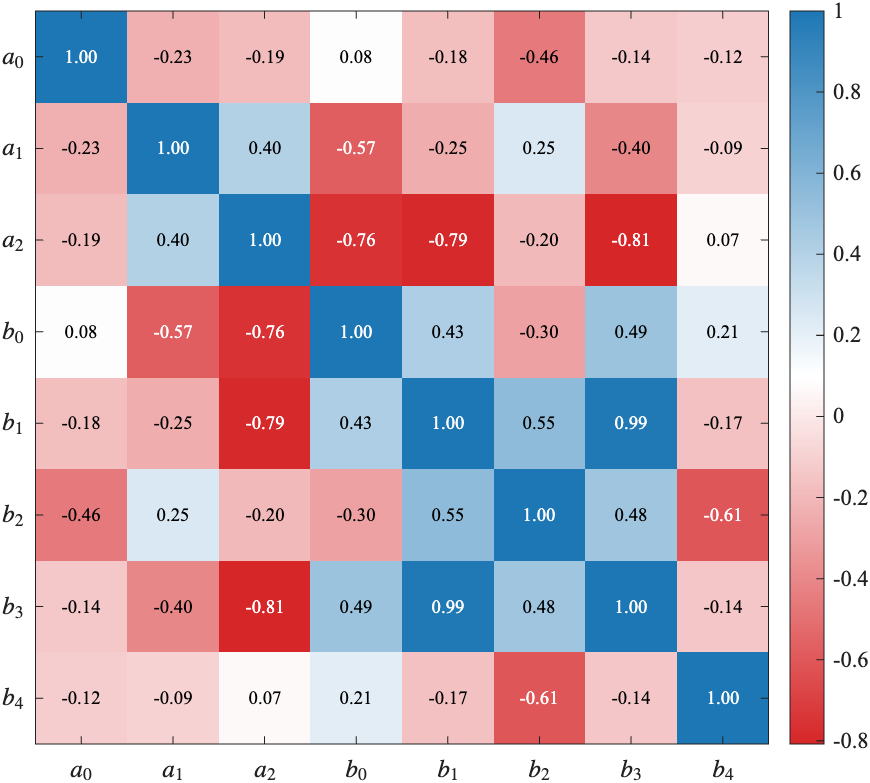}
\caption{Correlation coefficients matrix in Example \hyperref[exa2]{2}.
}
\label{fig:CorrTri}
\end{figure}

\section*{Acknowledgments}
\addcontentsline{toc}{section}{Acknowledgments}
The work described in this paper was supported by the NSF of China (12271151), the NSF of Hunan (2020JJ4166), and the Postgraduate Scientific Research Innovation Project of Hunan Province (CX20240364).

\section*{Appendix}\label{sec7}
\textbf{The proof of Theorem \ref{dl1}:}
\vspace{2mm}Let $\overrightarrow{\varepsilon}=(\varepsilon_{1},\varepsilon_{2},\varepsilon_{3})\in \mathbb{R}^3$, $g=\sum_{k=1}^{3}\varepsilon_{k}g_{k}$, $g_{k}\in C^{1.\alpha}(\partial\Omega)$. As \vspace{1mm}for the well-posedness results Theorem \ref{sdx}, it can be seen that when $|\overrightarrow{\varepsilon}|\rightarrow 0$, $u(\cdot,\varepsilon)\rightarrow 0$ and $u(\cdot,\varepsilon)|_{\varepsilon=0}=0$.

Note $u_{j}=u_{j}(x,\varepsilon)$, $j=1,2$, is the solution of
\begin{equation}\label{gjxxh1}
	\begin{cases}
		\Delta u_{j}+k^2(1+\alpha_{j}(x))u_{j}+k^2\beta_{j}(x)u_{j}^3=0 &in \,\Omega,\vspace{1mm}\\
		\dfrac{ \partial u_{j} }{ \partial \nu }=\varepsilon_{1}g_{1}+\varepsilon_{2}g_{2}+\varepsilon_{3}g_{3} &on \,\partial\Omega.
	\end{cases}
\end{equation} 	

Let $V_{j}^{(l)}=\partial_{\varepsilon_{l}}u_{j}|_{\varepsilon=0}$, $l=1,2,3$. Differentiating (\ref{gjxxh1}) with respect to $\varepsilon_{l}$, taking $\varepsilon=0$, and using $u_{j}(x,0)=0$, we get
\begin{equation}\label{xxfwtinnl}
	\begin{cases}
		\Delta V^{(l)}_{j}+k^2(1+\alpha_{j}(x))V^{(l)}_{j}=0 &in \,\Omega,\vspace{2mm}\\
		\dfrac{ \partial V^{(l)}_{j} }{ \partial \nu }=g_{l} &on \,\partial\Omega.
	\end{cases}
\end{equation} 	
Applying Theorem \ref{xxfwtdl} to (\ref{xxfwtinnl}), we have $\alpha_{1}(x)=\alpha_{2}(x)$ in $\Omega$. Therefore, we note $\alpha(x):=\alpha_{1}(x)=\alpha_{2}(x).$
And from Theorem \ref{sdx}, it can be seen that there exists a unique solution $V^{(l)}:=V_{1}^{(l)}=V_{2}^{(l)}.$

Next, we will discuss the uniqueness of determining $\beta(x)$.

Let $X_{j}=\partial^{3}_{\varepsilon_{1}\varepsilon_{2}\varepsilon_{3}}u_{j}|_{\varepsilon=0}$, differentiating (\ref{gjxxh1}) with respect to $\varepsilon_{1}$, $\varepsilon_{2}$, $\varepsilon_{3}$, taking $\varepsilon=0$, and using $u_{j}(x,0)=0$, we get

\begin{equation}
	\begin{cases}\label{V123}
		\Delta X_{j}+k^2(1+\alpha(x))X_{j}+6k^2\beta_{j}(x)V^{(1)}V^{(2)}V^{(3)}=0 &in \,\Omega,\vspace{2mm}\\
		\dfrac{ \partial X_{j} }{ \partial \nu }=0 &on \,\partial\Omega.
	\end{cases}
\end{equation} 	

Subtract equation (\ref{V123}) by taking $j=1$ and $j=2$ respectively, and then we have
\begin{equation}\label{j=12}
	\Delta X_{1}-\Delta X_{2}+k^2(1+\alpha(x))(X_{1}-X_{2})=-6k^2(\beta_{1}(x)-\beta_{2}(x))V^{(1)}V^{(2)}V^{(3)}.
\end{equation} 	
Multiplying both sides of (\ref{j=12}) by $V^{(0)}\in C^{2,\alpha}(\overline{\Omega})$ which satisfies
\begin{equation}
	\begin{cases}
		\Delta V^{(0)}+k^2(1+\alpha(x))V^{(0)}=0 &in \,\Omega,\vspace{2mm}\\
		\dfrac{ \partial V^{(0)}}{ \partial \nu }=h(x) &on \,\partial\Omega.
	\end{cases}
\end{equation} 	
where $h(x)\in C^{1,\alpha}(\partial\Omega)$. then we get
$$	(\Delta X_{1}-\Delta X_{2})V^{(0)}+(k^2(1+\alpha(x))V^{(0)})(X_{1}-X_{2})\vspace{1.5mm}$$
$$=\Delta (X_{1}-X_{2})V^{(0)}-(X_{1}-X_{2})\Delta V^{(0)}\vspace{2mm}$$
$$\hspace{3.7mm}=-6k^2(\beta_{1}(x)-\beta_{2}(x))V^{(1)}V^{(2)}V^{(3)}V^{(0)}. $$
And integrating in $\Omega$, there are
\begin{equation} \label{jfgreen}
\int_{{\Omega}} \Delta (X_{1}-X_{2})V^{(0)}-(X_{1}-X_{2})\Delta V^{(0)} dx=\int_{{\Omega}} -6k^2(\beta_{1}(x)-\beta_{2}(x))V^{(1)}V^{(2)}V^{(3)}V^{(0)} dx.
\end{equation} 	
Using Green's formula on the left side of (\ref{jfgreen}), we have
$$\begin{aligned}
\int_{{\Omega}} \Delta (X_{1}-X_{2})V^{(0)}-(X_{1}-X_{2})\Delta V^{(0)} dx
&=\int_{\partial\Omega} V^{(0)}\frac{\partial(X_{1}-X_{2})}{\partial\nu}-(X_{1}-X_{2})\frac{\partial V^{(0)}}{\partial\nu} dS\\
&=\int_{\partial\Omega} -(X_{1}-X_{2})\frac{\partial V^{(0)}}{\partial \nu} dS.
\end{aligned}$$

Since $N_{\alpha_{1},\beta_{1}}(g)=N_{\alpha_{2},\beta_{2}}(g)$, we have $u_{1}|_{\partial\Omega}=u_{2}|_{\partial\Omega}$. Then 
$$X_{1}=X_{2} \hspace{2mm} on \hspace{1mm} \partial\Omega. $$
Therefore,
\begin{equation}\label{V0123}
	\int_{{\Omega}} -6k^2(\beta_{1}(x)-\beta_{2}(x))V^{(1)}V^{(2)}V^{(3)}V^{(0)} dx=0,
\end{equation}
where $V^{(i)}\in C^{2,\alpha}(\overline{\Omega})$, $i=0,1,2,3$, satisfying $\Delta V^{(i)}+k^2(1+\alpha(x))V^{(i)}=0$ in $\Omega$.

Applying Proposition \ref{cmx>2} to (\ref{V0123}) and $k \neq 0$, we have
$$\beta_{1}(x)=\beta_{2}(x) \hspace{2mm} in \hspace{1mm}\Omega. $$
\hfill\textbf{$\square$}


\end{document}